\documentclass[reqno]{amsart}
\usepackage[T1]{fontenc}
\usepackage[latin1]{inputenc}
\usepackage[english]{babel}
\usepackage[babel]{csquotes}

\usepackage{cite}
\usepackage{amssymb}
\usepackage{amsmath}
\usepackage{amsthm}
\usepackage{latexsym}
\usepackage{graphicx}
\usepackage{mathrsfs}
\usepackage{mathrsfs}
\usepackage{bbm}
\usepackage{verbatim}
\usepackage[colorlinks=true, pdfstartview=FitV, linkcolor=blue, 
citecolor=blue, urlcolor=blue]{hyperref}
\usepackage[shortlabels,inline]{enumitem}

\usepackage[usenames,dvipsnames]{color} 
\newcommand{\UUU}{\color{black}}

\newcommand{\EEE}{\color{black}}
\newcommand{\disp}{\displaystyle}

\newtheorem{thm}{Theorem}[section]

\newtheorem{lem}[thm]{Lemma}
\newtheorem{prop}[thm]{Proposition}
\newtheorem{defin}[thm]{Definition}

\numberwithin{equation}{section}

\def\enne{\mathbb{N}}

\def\erre{\mathbb{R}}
\def\Rz{\mathbb{R}}

\def\P{\mathbb{P}}

\def\E{\mathop{{}\mathbb{E}}}

\def\cE{\mathscr{E}}
\def\cL{\mathscr{L}}
\def\cF{\mathscr{F}}

\def\cB{\mathscr{B}}
\def\eps{\varepsilon}
\def\cP{\mathscr{P}}

\def\OO{\mathcal{O}}
\renewcommand{\d}{{\mathrm d}}

\def\beq{\begin{equation}}
\def\eeq{\end{equation}}

\def\to{\rightarrow}
\def\wto{\rightharpoonup}
\def\wstarto{\stackrel{*}{\rightharpoonup}}
\def\embed{\hookrightarrow}

\def\norm #1{\left\|#1\right\|}

\def\sp #1#2{\left<#1,#2\right>}
\newcommand\ip\sp
\def\qv #1{\langle#1\rangle}
\def\qqv #1{\langle\!\langle#1\rangle\!\rangle}

\newcommand{\haz}{\widehat}
\newcommand{\epsi}{\eps}

\begin{document}
\title[Rate--independent stochastic evolution equations]
{Rate--independent stochastic evolution equations: parametrized solutions}

\author{Luca Scarpa}
\address[Luca Scarpa]{Department of Mathematics, Politecnico di Milano, 
Via E.~Bonardi 9, 20133 Milano, Italy.}
\email{luca.scarpa@polimi.it}
\urladdr{https://sites.google.com/view/lucascarpa}

\author{Ulisse Stefanelli}
\address[Ulisse Stefanelli]{Faculty of Mathematics,  University of Vienna, Oskar-Morgenstern-Platz 1, A-1090 Vienna, Austria,
Vienna Research Platform on Accelerating Photoreaction Discovery,
University of Vienna, W\"ahringer Str. 17, 1090 Vienna, Austria, $\&$ Istituto di Matematica Applicata e Tecnologie Informatiche ``E. Magenes'' - CNR, v. Ferrata 1, I-27100 Pavia, Italy}
\email{ulisse.stefanelli@univie.ac.at}
\urladdr{http://www.mat.univie.ac.at/$\sim$stefanelli}



\begin{abstract}
By extending to the stochastic setting the classical
vanishing viscosity approach
we prove the existence of suitably weak solutions of a class of
nonlinear stochastic evolution equation of rate-independent
type. Approximate solutions are obtained via viscous
regularization. Upon properly rescaling time, these approximations
converge to a parametrized  martingale solution of the
problem in rescaled time,
 where the rescaled noise is given by a general 
square-integrable cylindrical martingale with 
absolutely continuous quadratic variation.  In absence of jumps,
these are strong-in-time martingale solutions of the problem in the
original, not rescaled time. 
\end{abstract}

\maketitle

%%%%%%%%%%%%%%%%%%%%%%%%%%%%%%%%%%%%%%%%%%%%%%%%

\section{Introduction}
\setcounter{equation}{0}
\label{sec:intro}

Rate-independent models  arise
as limits of evolution systems featuring two extremely different timescales:
a very fast timescale related to internal relaxation, and a
much slower scale related to external loadings. In such a limit, the
internal time-relaxation scale is sent to zero, and the system
immediately relaxes in reaction to external stimuli. Examples of
rate-independent systems include dry friction, elastoplasticity,
brittle crack growth, damage, phase change, delamination
ferromagnetism, magneto- and piezoelasticity
\cite{Mielke-Roubicek15}. Rate-independence is indeed the signature of
 {\it hysteretic} phenomena. The
reader is referred to  \cite{Mielke-Roubicek15} and  
\cite{Brokate-Sprekels96,Krejci96,Visintin94} for some general theory
on the mathematical theory of hysteresis in connection with PDEs, as
well as for \UUU
a \EEE collection of applications. \UUU Some basic reference material
can be found in Section \ref{sec:deterministic} below. \EEE

This paper is concerned with the existence of suitably weak solutions to the Cauchy problem for the
abstract doubly nonlinear stochastic
evolution equation of {\it rate-independent} type 
\begin{equation}
  \label{eq0}
  A(\partial_tu^d)\, \d t + u^s \, \d W  + \partial \haz B(u)\, \d t
  \ni
  G(\cdot,u)\, \d W\,.
\end{equation}
The trajectory $u: \Omega \times [0,T] \to H$ is defined on the
stochastic basis $(\Omega,\cF,(\cF_t)_{t\in[0,T]},\P)$ and the bounded
time interval $[0,T]$ and takes values in  the  separable Hilbert space
$H$.

Solutions $u$ of equation \eqref{eq0} are asked to be {\it It\^o
processes} of the form 
\begin{equation}\label{ito_process}
  u(t)=u^d(t) + \int_0^tu^s(s)\,\d W(s), \quad t\in[0,T],
\end{equation}
where $u^d$ is an absolutely continuous process
and $u^s$ is a $W$-stochastically integrable process, with $W$ being 
a cylindrical Wiener process on a second separable Hilbert space
$U$. In particular, given the decomposition in \eqref{ito_process}, relation
\eqref{eq0} can be equivalently rewritten as the system
   \begin{equation}
  A(\partial_tu^d)  + \partial \haz B(u)
  \ni 0\,,\quad
  u^s=G(\cdot,u)\,.\label{eq0_bis}
  \end{equation}

The functional $\haz B$ is assumed to be convex lower semicontinuous
and to have compact sublevels, the
nonlinearity $G$ is taken to be suitably smooth.
The {\it rate-independent character} of the evolution ensues from
specifying 
the operator $A:H \to 2^H$ as
\begin{equation}
A(v):=
\left\{
  \begin{array}{ll}
    \disp v/\| v \|_H
    &\ \ \text{if} \ v \not =0 \\[1.5mm]
                    \{w \in H \: : \:\UUU\| w \|_H \EEE\leq 1  \},&\ \
                                                                   \text{if}
                                                                   \ \UUU v  =0  .    \EEE
  \end{array}
  \right.\label{A}
\end{equation}
\UUU We postpone to Section \ref{sec:deterministic} some discussion
about the applicative relevance of this choice. \EEE

The focus of this paper is on proving the existence of suitable
martingale (i.e.
probabilistic weak) solutions to \eqref{eq0}. The main
challenge is here to tame the doubly nonlinear nature of the
evolution in coordination with stochasticity. In fact, problem \eqref{eq0} is
well-known in case $A$ is linear. Well-posedness for analytically weak solutions is addressed by the classical variational theory 
by {Pardoux} \cite{Pard0,Pard} and {Krylov
\&  Rozovski\u \i},\cite{kr}, see also the  monographs \cite{dapratozab,pr}
for a general presentation. In the context of analytically strong solutions, 
existence for stochastic equations in the form \eqref{eq0} with
$A={\rm id}$
has been obtained by {Gess} \cite{Gess}. The reader is referred
also to  the monograph \cite{LiuRo} and
\UUU to the papers \EEE
\cite{mar-scar-diss,mar-scar-ref, mar-scar-div, scar-div} for other
well-posedness   results  
under more general conditions.
The doubly nonlinear case, namely the case where both $A$ and $\partial \haz B$
are
nonlinear, is more involved and, to date, has been solved only in the
{\it rate-dependent} setting of
\cite{ScarStef-SDNL2}, where $A$ is asked to be sublinear but
nondegenerate \UUU (namely, $A$ is of linear growth and coercive on $H$).\EEE
The reader is also referred to 
\cite{SWZ18,ScarStef-SDNL} for the existence theory of a distinct class of
doubly nonlinear stochastic evolution equations, where the term
$A(\partial_t u^d)\, \d t + u^s \, \d W $ is replaced by $\d
\hat A(u)$, with $A$ again sublinear and nondegenerate.

We tackle the existence problem for \eqref{eq0} by means of the so-called {\it vanishing-viscosity}
approximation: the $0$-homogeneous term $A(\partial_t u^d)\, \d t $ is
replaced by the nondegenerate term $\eps
\partial_t u^d \, \d t +  A(\partial_t u^d)\, \d t $, depending on the
small parameter   $\eps>0$. The approximated problems are then solved
by means of the rate-dependent theory of \cite{ScarStef-SDNL2}. A solution of \eqref{eq0} is then
recovered by letting $\eps \to 0$. In the deterministic case, this approach is indeed classical and has been applied to a
variety of different rate-independent settings. Without claiming
completeness, we refer the reader to \cite{EM} as well as to
\cite{Bonaschi20,Fiaschi09,Minotti18,MRS09,MRS12} for a collection of
abstract results. Applications to plasticity \cite{Babadjian12,DalMaso08,DalMaso11}, fracture
\cite{Almi17,Cagnetti08,Knees08,Lazzaroni13,Negri08}, damage \cite{Crismale16,Knees15,Lazzaroni18}, 
adhesive contact \cite{Bonfanti96,Roubicek13}, delamination
\UUU \cite{Scala17}, \EEE
optimal control \cite{Wachsmuth17}, and topology optimization
\cite{Almi,Almi2} are also available. Recall that, although the artificial
viscous term disappears as $\eps \to 0$, the choice of the specific form
of viscosity actually influences the limit
\cite{Stefanelli09,Krejci-Monteiro20}. 

In order to pass to the vanishing-viscosity limit $\eps \to 0$, we extend to the
stochastic setting the by-now
classical {\it time-reparametrization} technique \UUU introduced by
Efendiev \&  Mielke \EEE 
\cite{EM,Mielke14} (see also \cite{Bonfanti96,Gastaldi98,Martins95} for some previous applications of
this technique in the context of dry friction with small
viscosity): One uses the rate-indepedent nature of the problem by
reparametrizing the viscosity solutions $u_\eps$ via their {\it
  arc-length} parametrizations $\hat t_\eps$, so that the reparametrized
solutions $\hat u_\epsi$ are regular in time, uniformly with
respect to $\eps$, and then pass to the
limit simultaneously in  $\hat u_\epsi$ and $\hat t_\eps$. This procedures identifies
so-called {\it parametrized solutions}, a weak solution concept
focusing on  graphs in $[0,T]\times H$, rather than on 
trajectories from $[0,T]$ to $H$. Within this setting, 
possible jumps between states 
are resolved as vertical segments in graph space
\cite[Sec. 3.8.1]{Mielke-Roubicek15}.  Our main result, Theorem
\ref{thm1}, states the existence of {\it parametrized martingale
  solutions} of \eqref{eq0}, see Definition \ref{def:rs}. These are 
weaker than {\it differential martingale solutions} in original time,
see Definition \ref{def:vs}. Still, the two concepts coincide in
absence of jumps, see Subsection \ref{ssec:main}.

With respect to the deterministic case, the time-reparametrization
procedure in the stochastic cases poses some distinct
challenges. The main difficulty is that 
the arc-length reparametrization of solutions to the $\eps$-viscous case
here is a stochastic process $\hat t_\eps:\Omega\times\erre_+ \to [0,T]$,
not a deterministic function, 
and one aims at rescaling the noise term $G(\cdot)\,\d W$ through $\hat t_\eps$
and \UUU studying \EEE the asymptotic behaviour of the rescaled noise.
This poses delicate measurability issues, especially in the direction of
characterizing the rescaled noise by suitable stochastic integration 
with respect to a time-rescaled cylindrical martingale (possibly not Wiener).
To begin with,
one is asked to check that the arc-length parametrization $\hat
t_\eps(\cdot,\tau):\Omega \to [0,T]$ is a bounded stopping time for
each $\tau\geq 0$ and $\eps>0$ (Lemma \ref{lem:t}),
and the whole rescaling procedure pivots around the
study of measurability properties of the family 
$\{\hat t_\eps(\cdot,\tau)\}_{\tau\geq0}$ of stopping times.
Moreover, the time-rescaling procedure inevitably 
demands a rescaling argument of the driving filtration as well,
through the above family of stopping times.
In addition, 
the corresponding time-rescaling of stochastic integrals calls for
handling integration with respect to general continuous
square-integrable cylindrical martingales, as opposed to the classical
integration with respect to cylindrical Wiener process, and optional stopping arguments.
Eventually, a further issue concerns stochastic compactness:
indeed, after the time-rescaling procedure one has to handle the 
convergence of the rescaled noises $\hat M_\eps=W(\hat t_\eps)$,
as well as \UUU that \EEE of their respective stochastic integrals, as the 
viscosity coefficient $\eps$ vanishes.

The paper is organized as follows. In Section \ref{sec:main} we
introduce notation, discuss preliminaries on cylindrical martingales,
and present our assumptions. The  notions of  differential and parametrized martingale
solutions \UUU are \EEE introduced and we 
state our main existence result, Theorem \ref{thm1}. The relation
between differential and parametrized martingale
solutions is detailed in Theorem \ref{thm0}, which is proved in
Section \ref{sec:sol}. The viscous regularization of \eqref{eq0} is
addressed in Section \ref{sec:limit}. Eventually, we pass to the
vanishing-viscosity limit in Section \ref{sec:ri}, providing a proof of
Theorem \ref{thm1}.

\EEE

%%%%%%%%%%%%%%%%%%%%%%%%%%%%%%%%%%%%%%%%%%%%%%%%

\section{Preliminaries, setting, and main results}
\label{sec:main}
We introduce here the general framework of the paper, 
recall some useful theoretical preliminaries, 
state the assumptions, and state the main results.

\subsection{Notation}
\label{ssec:notation}
\UUU Let $X$ be an arbitrary Banach space, \EEE $s,r\in[1,+\infty]$, and $T>0$.
\UUU We fix a \EEE filtered probability space
$(\Omega,\cF,(\cF_t)_{t\geq 0},\P)$ \UUU (here and in the following, one
can consider the case $t\in[0,T]$, as well). This is a probability space
$(\Omega,\cF,\P)$ over the $\sigma$-algebra $\cF$, together with a filtration
$(\cF_t)_{t\geq0}\subset \cF$, namely, an increasing collection of  
of $\sigma$-algebras, i.e.,~$\cF_s\subseteq\cF_t$ for all
$0\leq s\leq t $.
We say that $(\Omega,\cF,(\cF_t)_{t\geq0},\P)$ satisfies the
{\it usual conditions} if the filtration 
$(\cF_t)_{t\geq 0}$ is saturated ($\cF_0$ contains all 
negligible sets, i.e.,~the sets
$A\in\cF$ with $\P(A)=0$)
and right-continuous ($\cF_s=\cap_{t>s}\cF_t$ for all $s\geq0$). 

The \EEE usual symbols $L^s(\Omega; X)$ and $L^r(0,T; X)$
indicate the spaces of strongly measurable Bochner-integrable functions 
on $\Omega$ and $(0,T)$, respectively. \UUU Let $(X,\cE)$ be a
measurable space.  A {\it stochastic process} with values in $X$ 
is a function $u:\Omega\times[0,\infty)\to X$.
We say that $u$ is {\it measurable} when 
it is $\cF\otimes\cB([0,\infty))/\cE$-measurable ($\cB([0, \infty))$ are the
Borel sets of $[0, \infty))$). In this case, the Tonelli theorem
ensures that $u(\omega,\cdot):[0,\infty)\to X$
is $\cB([0,\infty))/\cE$-measurable for 
$\P$-almost every $\omega\in\Omega$. Hence, 
every measurable stochastic process $u$
may be
also seen as a random variable 
with values in the space $L^0(0,T;X)$
of measurable $X$-valued functions, 
i.e., $u:\Omega\to L^0(0,T;X)$.

We say that  $u$ is {\it adapted} to the 
filtration $(\cF_t)_{t\geq 0}$ when
$u(\cdot,t):\Omega\to X$ is 
$\cF_t/\cE$-measurable for every $t \geq0$ and that 
$u$ is {\it progressively measurable} when
the restriction $u_{|\Omega\times[0,t]}$ 
is $\cF_t\otimes\cB([0,t])/\cE$-measurable
for every $t\geq 0$.
The $\sigma$-algebra on $\Omega\times[0, \infty)$
generated by the progressively measurable 
processes is called {\it progressive $\sigma$-algebra}
and is denoted by $\cP$.
Clearly, progressive measurability 
implies adaptedness.
Given a stochastic process $u$, the 
filtration $(\cF_t^u)_{t\geq0}$
generated by $u$, or {\it natural filtration}
associated to $u$,
is defined as the smallest   filtration 
that makes $u$ adapted.

\EEE
When $s,r\in[1,+\infty)$ we use the symbol 
$L^s_\cP(\Omega;L^r(0,T; X))$ to specify 
that measurability is intended with respect to 
the corresponding progressive $\sigma$-algebra $\cP$.
In the special case where $s\in(1,+\infty)$, $r=+\infty$, and $X$ is separable and reflexive,
we explicitly set 
\begin{align*}
  &L^s_w(\Omega; L^\infty(0,T; X^*)):=
  \big\{v:\Omega\to L^\infty(0,T; X^*) \text{ weakly*  measurable}\nonumber\\
&\hspace{48mm} \text{with} \ 
  \E\norm{v}_{L^\infty(0,T; X^*)}^s<\infty
  \big\}\,,
\end{align*}
and recall the identification \cite[Thm.~8.20.3]{edwards}
\[
L^s_w(\Omega; L^\infty(0,T; X^*))=
\left(L^{\frac{s}{s-1}}(\Omega; L^1(0,T; X))\right)^*\,.
\]

Given any pair of Banach spaces $X_1$ and $X_2$, 
we use the symbols $\cL(X_1,X_2)$, $\cL_s(X_1,X_2)$, and $\cL_w(X_1,X_2)$
for the spaces of linear continuous operators from $X_1$ to $X_2$
endowed with the norm topology, the strong operator topology,
and the weak operator topology, respectively.
When $X_1$ and $X_2$ are Hilbert spaces, we denote by \UUU
$\cL^2(X_1,X_2)$ the {\it Hilbert-Schmidt} operators from $X_1$ to
$X_2$. Namely,  $L\in\cL^2(X_1,X_2)$  if there exists 
a complete orthonormal system $(u_j)_j$ of $X_1$
such that 
\[
  \sum_{j=0}^\infty\norm{Lu_j}_{X_2}^2<+\infty\,.
\]
One can check that this definition is independent of the  
complete orthonormal system of $X_1$. 
The linear space   $\cL^2(X_1,X_2)$ is a separable 
Hilbert space with respect to 
the scalar product and norm
\[
  (L,T)_{\cL^2(X_1,X_2)}:=
  \sum_{j=0}^\infty(Lu_j, Tu_j)_{X_2}\,,
  \qquad  \norm{L}_{\cL^2(X_1,X_2)}^2:=
   \sum_{j=0}^\infty\norm{Lu_j}_{X_2}^2
   \,,
  \]
respectively. 
  
An operator $L\in \cL^2(X_1,X_2)$ is called {\it nuclear}
if there exist sequences $(u_j)_j\subset B_1^{X_1}$
and $(v_j)_j\subset B_1^{X_2}$ in the closed 
unit balls of $X_1$ and $X_2$, respectively, 
and $\alpha\in\ell^1$ such that 
\[
  Lx=\sum_{j=0}^\infty\alpha_j(x,u_j)_{X_1} v_j \quad
  \forall\,x\in {X_1}\,.
\]
The linear space of nuclear operators 
from $X_1$ to $X_2$ is denoted by $\cL^1(X_1,X_2)$, and 
endowed with the norm
\[
  \norm{L}_{\cL^1(X_1,X_2)}:=
  \inf\left\{
  \norm{\alpha}_{\ell_1}:\ \ 
  Lx=\sum_{j=0}^\infty\alpha_j(x,u_j)_{X_1} v_j \quad
  \forall\,x\in X_1
  \right\} \,.
\]

The important case $X_1=X_2$ requires some additional comment.
In fact, given an operator 
$L\in\cL(X_1,X_1)$
one has that $L^*L\in\cL(X_1,X_1)$ is 
nonnegative and symmetric and one defines   $|L|:=(L^*L)^{1/2}\in\cL(X_1,X_1)$ as the unique 
symmetric linear operator such that $|L|^2=L^*L$.
In this case, one can show that $L\in\cL^1(X_1,X_1)$
if and only if there exists a complete 
orthonormal system $(e_j)_j$ of $X_1$ such that
\[
  \operatorname{Tr}|L|:=
  \sum_{j=0}^\infty(|L|e_j,e_j)_{X_1}<+\infty\,.
\]
Indeed, given such $(e_j)_j$ one can choose
for example $u_j:=e_j$, 
$v_j:=\norm{Le_j}_{X_1}^{-1}Le_j$, and 
$\alpha_j:=\norm{Le_j}_{X_1}$.
In this case, the quantity $\operatorname{Tr}|L|$
is called {\it trace} of $|L|$, its definition 
does not depend on the basis $(e_k)_k$, and
it holds that 
$\operatorname{Tr}|L|= \norm{L}_{\cL^1(X_1,X_1)}$.
The space $\cL^1(X_1,X_1)$ is called 
space of {\it trace-class} operators on $X_1$.
Note that if $L\in\cL^1(X_1,X_1)$ is also symmetric, 
then we have $|L|=L$, and $\operatorname{Tr}L$
is the trace of $L$.

\EEE

\subsection{Preliminaries on cylindrical martingales}
\label{ssec:cyl}
We collect here some preliminaries and definitions 
about the general theory of Hilbert-space-valued cylindrical martingales
and the respective stochastic integration.

For any separable Hilbert space $U$ and 
every filtered probability space $(\Omega,\cF,(\cF_t)_{t\geq0},\P)$
satisfying the usual conditions, the classical formal 
representation of a {\it $U$-cylindrical Wiener process} is 
\beq
  \label{W}
  W=\sum_{k=0}^\infty \beta_k e_k\,,
\eeq
where $(e_k)_k$ is a complete orthonormal system of $U$ and 
$(\beta_k)_k$ is a sequence of independent real-valued Brownian motions.
The series above is not convergent in $U$
unless $U$ has finite dimension. Hence, in order to 
rigorously define the concept of $U$-cylindrical Wiener process
it is instrumental to interpret
  the series \eqref{W} in a suitably enlarged space $U_1$ 
such that $U\embed U_1$ continuously and the inclusion $\iota:U\to U_1$
is Hilbert-Schmidt. In fact, we first define a new norm $\norm{\cdot}_{U_1}$
on $U$ as
\[
  \norm{u}_{U_1}:=\left(\sum_{k=0}^\infty\frac1{(1+k)^2}|(u,e_k)_U|^2\right)^{1/2}\,, \qquad u\in U\,.
\]
Noting that \UUU $(U,\norm{\cdot}_{U_1})$ \EEE is clearly not complete
\UUU in \EEE  $U$, we define $U_1$
as the abstract closure of $U$ with respect to the topology induced by
$\norm{\cdot}_{U_1}$. This
can be equivalently characterised as
the space of formal infinite linear combinations of $(e_k)_k$
with finite norm $\norm{\cdot}_{U_1}$, i.e.,
\[
  U_1=\left\{u=\sum_{k=0}^\infty\alpha_k e_k: \; 
  \sum_{k=0}^\infty\frac{\alpha_k^2}{(1+k)^2}<+\infty\right\}\,.
\]
With this choice, $U_1$ is a separable Hilbert space and the 
natural inclusion $\iota:U\to U_1$ is Hilbert-Schmidt: indeed, 
by the Parseval identity we have 
\[
  \sum_{k=0}^\infty\norm{\iota (e_k)}_{U_1}^2
  =\sum_{k=0}^\infty\norm{e_k}_{U_1}^2
  =\sum_{k=0}^\infty\sum_{i=0}^\infty\frac{|(e_k,e_i)_U|^2}{(1+i)^2}
  =\sum_{k=0}^\infty\frac{\norm{e_k}_U^2}{(1+k)^2}
  =\sum_{k=0}^\infty\frac1{(1+k)^2}<+\infty\,.
\]
Furthermore, the operator $Q_1:=\iota\circ\iota^*\in \cL^1(U_1,U_1)$
is nonnegative and symmetric, 
the infinite sum \eqref{W} is convergent in $U_1$,
and $W$ is a classical $Q_1$--Wiener process on $U_1$.
The following definition is then natural.
\begin{defin}
\UUU A $Q_1$--Wiener process on $U_1$ is called \emph{$U$-cylindrical
  Wiener process}. A $U_1$--valued martingale is called
\emph{$U$-cylindrical martingale}. A $U_1$--valued martingale is
said to be \emph{continuous} (\emph{square-integrable}, respectively) if
it is continuous with values in $U_1$ (square-integrable, respectively).
  \EEE
  % A $U$-cylindrical Wiener process is a $Q_1$--Wiener process on $U_1$.
  % More generally, we say that $M$ is a 
  % $U$-cylindrical martingale if $M$ is a $U_1$--valued martingale.
  % We also say that $M$ is a continuous (respectively square-integrable)
  % $U$-cylindrical martingale
  % if $M$ is a continuous (respectively square-integrable) martingale with values in $U_1$.
\end{defin}
 
$U$-Cylindrical Wiener processes are continuous square--integrable martingales on $U_1$
with quadratic variation and tensor quadratic variation given by, respectively, 
\[
  \qv{W}(t)=\operatorname{Tr}(Q_1)t\,, \qquad
  \qqv{W}(t)=Q_1 t\,, \qquad t\geq0\,,
\]
\UUU see \cite[Prop. 2.5.2, p.~50]{LiuRo}. \EEE
More generally, an important class of $U$-cylindrical martingales that 
will play a central role in the paper
are the continuous square--integrable martingales $M$ on $U_1$ with 
tensor quadratic variation 
\beq
  \label{qv_M}
  \qqv{M}(t)=Q_1 \alpha(t)\,, \qquad t\geq0\,,
\eeq
where $\alpha:\Omega\times\erre_+\to\erre$ is a 
Lipschitz-continuous nondecreasing adapted process with $\alpha(0)=0$.

The stochastic integral with respect to such $U$-cylindrical martingales
is defined in terms of the usual $L^2$--stochastic integral 
in the space $U_1$: for details we refer the reader to M\'etivier
\cite[\S~21--22]{metivier}.
For any separable Hilbert space $K$,
the natural class of stochastically-integrable processes
with respect to $M$ is the space of progressively measurable processes
$B$
with values in $\cL^2(Q_1^{1/2}(U_1), K)$ such that 
\[
  \E\int_0^T\norm{B(s)Q_1^{1/2}}_{\cL^2(U_1, K)}^2\,\d \alpha(s)<+\infty.
\]
Exploiting the definitions of $\iota$ and $Q_1$
it is immediate to check that $Q_1^{1/2}(U_1)=\iota(U)$, so that 
$\cL^2(Q_1^{1/2}(U_1), K)=\cL^2(U,K)$, and 
$\|BQ_1^{1/2}\|_{\cL^2(U_1, K)}=\|B\|_{\cL^2(U,K)}$
for every $B\in\cL^2(U,K)$. Consequently, the stochastic integral with respect 
to a $U$-cylindrical martingale with quadratic variations as in \eqref{qv_M} is 
well-defined on the space
\begin{align*}
  &L^2_\cP(\Omega; L^2(0,T, \d\alpha; \cL^2(U,K)))\\
  &:=
  \left\{B:\Omega\times[0,T]\to\cL^2(U,K)\text{ progr.~meas.:}\quad
  \E\int_0^T\norm{B(s)}_{\cL^2(U, K)}^2\,\d \alpha(s)<+\infty\right\}\,.
\end{align*}
Let us point out that the Lipschitz-continuity of $\alpha$
ensures that $\d\alpha(\omega,\cdot)$ is absolutely continuous with respect to the
Lebesgue measure: this is why $B$ only needs be progressively measurable 
% and not necessarily predictable,
and the integrability condition above reads equivalently 
\[
  \E\int_0^T\norm{B(s)}_{\cL^2(U, K)}^2\alpha'(s)\,\d s<+\infty\,.
\]
We recall for completeness that for any such $B$ the stochastic integral is defined 
as the $L^2(\Omega)$-limit of the usual stochastic integral defined on (predictable)
elementary \UUU processes, \EEE namely
\[
  \int_0^T B(s)\,\d M(s)=
  \lim_{N\to\infty}\sum_{i=0}^{N-1}\mathbbm{1}_{A_i}\xi_i(M(t_{i+1})-M(t_i))
  \qquad\text{in } {L^2(\Omega)}\,,
\]
where $0=t_0<\ldots<t_N=T$ is a partition of $[0,T]$
with $\lim_{N\to\infty}\sup_{i=0,\ldots,N-1}|t_{i+1}-t_{i}|=0$,
$A_i\in\cF_{t_i}$, $\xi_i\in\cL(U_1,K)\cap\cL^2(U,K)$ for every $i=0,\ldots, N$.

\subsection{Variational setting and assumptions}
\label{ssec:assumptions}

We state now the main assumptions on the data of the problem.
These will hold throughout the whole paper, and will not 
be explicitly recalled.

\begin{itemize}\setlength\itemsep{1em}
  \item {\bf Assumption H:}  
  $T>0$ is the final reference time, 
$V$ is a separable reflexive Banach space and $H$ is a
Hilbert space such that $V\embed H$ continuously, densely, and compactly.
The space $H$ is identified with its dual $H^*$, so that we have the
Gelfand triple $
  V\embed H\embed V^*$. We 
 require the existence of a sequence of 
regularising operators $(R_\lambda)_{\lambda>0}\subset\cL(H,V)$ 
such that $R_\lambda\to I$ in $\cL_s(V,V)$ as $\lambda\searrow0$.
Finally, we ask for another separable reflexive Banach space 
$E^*$ such that $H \subset E^*$ compactly and
\begin{equation}
  \label{eq:appro}
 \exists\, 
 c_E>0\,, \ \alpha \in (0,1): \quad   
 \| z \|_H \leq \eta \| z \|_V + \frac{c_E}{\eta^\alpha} \| z
   \|_{E^*} \quad \forall\, z \in V, \ \forall\, \eta>0\,.
\end{equation}
The existence of such $(R_\lambda)_\lambda$ and $E^*$ 
are to be regarded as a structural
compatibility assumptions, and are satisfied in most application
scenarios.  For example, in most applications one has 
$H=L^2(\OO)$, with $\OO$ being a sufficiently regular bounded domain,
and $V=W^{s,p}_0(\OO)$ for some $s>0$ and $p\geq2$: here, 
a natural candidate for $E^*$ could be the dual space $W^{-r,q}(\OO)$
with $q=p/(p-1)$ and $r\in(0,s)$.
Of course, different choices 
of boundary conditions can be handled similarly.

\item {\bf Assumption U0:} $u_0 \in V$.

\item{\bf Assumption A:} 
         Let $A:=\partial \| \cdot \|_H:H\to2^H$, so that $A(u) =
        u/\| u \|_H$ if $u \not =0$ and $A(0)=\{ v \in H  :  \|v
        \|_H\leq 1\}$.
        	We further define the lower semicontinuous convex functional  
	 \[
	\Psi_{\norm{\cdot}}:H\to[0,+\infty]\,, \qquad
	\Psi_{\norm{\cdot}}(z):=
	\begin{cases}
	\| z\|_H \quad&\text{if } \norm{z}_H\leq1\,,\\
	+\infty \quad&\text{if } \norm{z}_H>1\,,
	\end{cases}
	\]
\item{\bf Assumption B:} $\Phi:V\to[0,+\infty)$
	is  twice G\^ateaux--differentiable, with $\Phi(0)=0$. We
	 set $B:=D_{\mathcal G}\Phi:V\to V^*$
         (G\^ateaux--differential), and assume that
	  \[
 	 D_{\mathcal G}B\in C^0(V; \cL_w(V,V^*))\,.
 	 \]
 	 Moreover, we assume that there exist constants $c_B, C_B>0$, 
	 $c_B'\geq0$, and $p\geq2$ such that, for every $z_1,z_2,z\in V$,
 	 \begin{align*}
  	  \ip{B(z_1)-B(z_2)}{v_1-v_2}_{V^*,V} &\geq c_B\norm{z_1-z_2}_V^p-
	  c_B'\norm{z_1-z_2}_H^2\,,\\
  	  \norm{D_{\mathcal G}B(z)}_{\cL(V,V^*)}  &\leq 
  	  C_B\left(1 + \norm{z}_V^{p-2}\right)\,,
 	 \end{align*}
	where $c_B'=0$ if $p=2$. We set for convenience $q:=
        p/(p-1)$ \UUU and denote by $D(B)$ the effective domain of $B$ on $H$.  
        Eventually, we suppose the existence of a reflexive Banach 
        space $Z\subset D(B)$ and a sequence $(P_n)_n\subset\cL(H,Z)$ such that 
        $P_nx\to x$ in $H$ (resp.~$V$) for every $x\in H$ (resp.~$V$),
        $B(P_nx)\to B(x)$ for every $x\in D(B)$, and
        \begin{align*}
        \norm{B(z)}_H&\leq C_B(1+\norm{z}^{p-1}_Z) \quad\forall\,z\in Z\,,\\
        \norm{B(P_nz)}_H&\leq C_B(1+\norm{B(z)}) \quad\forall\,z\in D(B)\,.
        \end{align*}
        \EEE
		
\item  {\bf Assumption G:} $U$ is a separable Hilbert space, 
	$G:[0,T]\times H\to \cL^2(U,H)$  is measurable,
	\UUU$G([0,T]\times V)\subset\cL^2(U,V)$, \EEE
  and there exist two constants $C_G>0$ and $\alpha_G\in(0,1]$ such that,
  for all $z_1,z_2,z_3\in H$, $z\in V$, and $t_1,t_2,t\in[0,T]$ it holds that
  \begin{align*}
    \norm{G(t_1, z_1)-G(t_2,z_2)}_{\cL^2(U,H)}
    &\leq C_G\left(\norm{z_1-z_2}_H
    +|t_1-t_2|^{\alpha_G}\right)\,,\\
    \norm{G(t,z_3)}_{\cL^2(U,H)}&\leq
    C_G\left(1+\norm{z_3}_V^\nu\right)\,,\\
    \UUU \norm{G(t,z)}_{\cL^2(U,V)}&\leq
    C_G\left(1+\norm{z}_V\right)\,,\EEE
  \end{align*}
  where $\nu=1$ if $p>2$ and $\nu\in(0,1)$ if $p=2$.  
  Moreover, the operator 
   $L:[0,T]\times V\to \cL^1(H,H)$ given by
  \[
  L(\cdot,v)=G(\cdot,v)G(\cdot,v)^*D_{\mathcal G}B(v)\,,
  \qquad v\in V\,,
  \]
  satisfies
  \begin{align*}
    &L\in C^0([0,T]\times V; \cL^1(H,H))\,,\\
    &\norm{L(\cdot,v)}_{\cL^1(H,H)}\leq  C_G(1+\norm{v}_V^p)
    \quad\forall\,v\in V\,.
  \end{align*}
  Note that this last requirement implies that 
  for every $t\in[0,T]$ and $v\in V$, not only is the operator 
  $G(t,v)G(t,v)^*D_{\mathcal G}B(v)$ in $\cL(V,V)$, 
  as it would normally be by the assumptions on $B$ and $G$,
  but it also
  extends to a trace--class operator in $\cL^1(H,H)$ whose growth 
  is suitably controlled.
\end{itemize}

An additional term $F(t,u) \, \d t$
on the right-hand side of \eqref{eq0} may also be considered without
particular difficulty. We nevertheless 
prefer to retain the only forcing  $G(t,u) \, \d W$ for the sake of
simplicity. Note however that we cannot simplify further by letting
the forcing to be independent of time, for the rate-independent nature of \eqref{eq0} would
then prevent evolution. 

\subsection{Main results}
\label{ssec:main}
In order to specify our results, let us first discuss some different notions   
of solution for problem \eqref{eq0_bis}. \UUU Also in the
deterministic case, rate-independent systems can be tackled under a
variety of different
approaches. Besides differential formulations, which are sometimes
precluded due to nonsmoothness, various weaker variational or
approximation settings heve been set forth
\cite{Mielke-Roubicek15}. Minimally, one should mention the theory of
{\it energetic solvability} \cite{MielkeTheil,MielkeTheilLevitas},
which is based on the reformulation of the
rate-independent system in terms of an energy balance and of a
quasistationary stability requirement, as well as the already mentioned {\it vanishing
  viscosity} approach, hinging on a limiting process on more
time-regular approximations. 
The reader is referred to \cite{Mielke-Roubicek15} and
\cite{MRS09,Negri10,Roubicek15,Stefanelli09} for a comparison between
these and other \EEE notions of solutions of
rate-independent systems in the deterministic setting.

Let us start by introducing a strong-in-time solution notion, solving
\eqref{eq0} almost everywhere in time, up to changing the underlying
stochastic basis. 

\begin{defin}[Differential martingale solution]
  \label{def:vs}
  A \emph{differential martingale solution} of problem \eqref{eq0_bis} is a family 
  consisting of a filtered probability space
  $(\Omega, \cF,(\cF_t)_{t\in[0,T]},\P)$ satisfying the usual conditions, 
  a $U$-cylindrical Wiener process $W$ on it,
  and a triple $(u, u^d, v)$ with 
   \begin{align*}
  & u \in L^p_{  \cP}( \Omega; C^0([0,T]; H))\cap 
  L^p_{ w}( \Omega; L^\infty(0,T; V))\,,
  && u^d \in L^1_{  \cP}( \Omega; W^{1,1}(0,T; H))\,,\\
  & v \in L^1_{  \cP}( \Omega; L^1(0,T, H))\,,
  &&B( u) \in L^1_{  \cP}( \Omega; L^1(0,T; H))\,,
  \end{align*}
  such that
  \begin{align}
  \label{eq1_var}
  & u(t) = u_0 + \int_0^t\partial_t u^d(s)\,\d s 
  + \int_0^tG(s, u(s))\,\d   W(s)
  &&\forall\,t\in[0,T]\,,\quad  \P\text{-a.s.}\,,\\
  \label{eq2_var}
  & v(t) + B( u(t)) = 0
  &&\text{for a.e.~$t\in(0,T)$}\,,\quad  \P\text{-a.s.}\,,\\
  \label{incl}
  & v(t) \in A(\partial_t  u^d(t)) 
  &&\text{for a.e.~$t\in(0,T)$}\,,\quad  \P\text{-a.s.}
  \end{align}
\end{defin}

Differential martingale solutions cannot be expected to exist in
general, for the nonconvex nature of the potential $\Phi$ may induce
jumps in time. One is hence forced to consider some weaker notion of
solution instead. Among the many options (see the discussion in 
\cite{Mielke-Roubicek15} in the deterministic case) we focus on a
stochastic version of {\it parametrized} solutions. These result by extending
trajectories in graph space, taking into account the rate-independent
nature of the problem. By reparametrizing time via a specific choice
for the arc-length, parametrized solutions are uniformly Lipschitz
continuous. In particular, they do not jump in \UUU rescaled \EEE time. More, precisely,
we introduce the following.

\begin{defin}[Parametrized martingale solution]
  \label{def:rs}
  A \emph{parametrized martingale solution} of problem \eqref{eq0_bis} is a family 
  consisting of a filtered probability space
  $(\Omega, \hat \cF,(\hat \cF_\tau)_{\tau\geq0},\P)$ satisfying the usual conditions, 
  a continuous square-integrable $U$-cylindrical martingale $\hat M$ on it,
  and a quadruplet $(\hat u, \hat u^d, \hat t, \hat v)$ such that, for all $\hat T>0$,
   \begin{align*}
  & \hat u \in L^p_{ \hat\cP}( \Omega; C^0([0,\hat T]; H))\cap 
  L^p_{ w}( \Omega; L^\infty(0,\hat T; V))\,,
  && (\hat u^d,\hat t) \in L^\infty_{  \hat\cP}( \Omega; 
  W^{1,\infty}(0,\hat T; H\times\erre))\,,\\
  & \hat v \in L^1_{  \hat\cP}( \Omega; L^1(0,\hat T, H))\,,
  &&B( \hat u) \in L^1_{  \hat\cP}( \Omega; L^1(0,\hat T; H))\,,
  \end{align*}
  where 
  \begin{align}
  \label{hat_t}
  &\hat t(0)=0\,,\qquad0\leq\hat t(\tau)\leq T
    &&\forall\,\tau\in[0,\hat T]\,, \qquad\P\text{-a.s.}\,,\\
  \label{eq_hat}
  &\hat t'(\tau) + \norm{\partial_t \hat u^d(\tau)}_H=1  \ \
    \text{and} \ \ \hat t'(\tau)\geq0 
  &&\text{for a.e.~$\tau\in(0,\hat T)$}\,, \quad\P\text{-a.s.}\,,\\
  \label{quadratic}
  &\qqv{\hat M}(\tau)=Q_1\hat t(\tau)\,,
  &&\forall\,\tau\in[0,\hat T]\,,\quad\P\text{-a.s.}\,,\\
  \label{eq1_var'}
  & \hat u(\tau) = u_0 + \int_0^\tau\partial_t \hat u^d(s)\,\d s 
  + \int_0^{\tau}G(\hat t(\sigma), \hat u(\sigma))\,\d  \hat M(\sigma)
  &&\forall\,\tau\in[0,\hat T]\,,\quad  \P\text{-a.s.}\,,\\
  \label{eq2_var'}
  & \hat v(\tau) + B( \hat u(\tau)) = 0
  &&\text{for a.e.~$\tau\in(0,\hat T)$}\,,\quad  \P\text{-a.s.}\,,\\
  \label{incl'}
  & \UUU \hat v \EEE (\tau) \in \partial\Psi_{\norm{\cdot}}(\partial_t  \hat u^d(\tau)) 
  &&\text{for a.e.~$\tau\in(0,\hat T)$}\,,\quad  \P\text{-a.s.}
  \end{align}
\end{defin}

The notion of parametrized martingale solution is strictly weaker than
that of differential martingale solution. An
example in this direction can be found already in the deterministic
case in  
\cite[Ex. 4.4]{EM}.  Still, the two notions coincide, in
case solutions in original time do not jump. More precisely, we
have the following.

\begin{thm}[Differential vs.~parametrized martingale solutions]
 \label{thm0}
 The following statements hold.
 \begin{itemize}
 \item[(i)]
 Let $(\Omega, \cF,(\cF_t)_{t\in[0,T]},\P, W, u, u^d, v)$ be a 
 differential martingale solution for \eqref{eq0_bis}.
 Then, there exist a filtration
 $(\hat\cF_\tau)_{\tau\geq0}$ satisfying the usual conditions
 and a stochastic process
 \[
 \hat t \in \bigcap_{\hat T>0}L^\infty_{\hat\cP}( \Omega; 
 W^{1,\infty}(0,\hat T; \erre))
 \]
 such that,
 setting $\hat u:=u\circ\hat t$, $\hat u^d:=u^d\circ\hat t$, 
 $\hat v:=v\circ \hat t$, and
 $\hat M:=W\circ\hat t$, it holds that
 $(\Omega, \hat\cF,(\hat\cF_t)_{t\in[0,T]},\P, \hat M, \hat u, \hat u^d, \hat v, \hat t)$
 is a parametrized martingale solution for problem \eqref{eq0_bis}.
 \item[(ii)]
 Let $(\Omega, \hat\cF,(\hat\cF_\tau)_{\tau\geq0},\P, \hat M, \hat u, \hat u^d, \hat v, \hat t)$
 be a parametrized martingale solution for problem \eqref{eq0_bis}
 and let $\hat S$ be a finite stopping time such that 
 $\P\{\hat t(\hat S)=T\}=1$ and
 $\|\partial_t\hat u^d\|_H<1$ almost everywhere in $[\![0,\hat
 S]\!] := \{(\omega,t) \in \Omega \times (0,T) \ : \ 0 \leq t \leq \hat
 S(\omega)\}$.
 Then, there \UUU exists \EEE a filtration
 $(\cF_t)_{t\in[0,T]}$ satisfying the usual conditions and
 $\hat t_{|[\![0,\hat S]\!]}$ has an almost surely continuous inverse
 \[
 \tau:\Omega\times[0,T]\to\erre_+
 \]
 such that, setting $u:=\hat u\circ\tau$, $u^d:=\hat u^d\circ \tau$, $v:=\hat v\circ \tau$,
 and $W:=\hat M\circ \tau$,
 it holds that 
 $(\Omega, \cF,(\cF_t)_{t\in[0,T]},\P, W, u, u^d, v)$ is a 
 differential martingale solution for problem \eqref{eq0_bis}.
 \end{itemize}
\end{thm}
A proof of this statement is provided in Section \ref{sec:sol} below.

% In fact, the notion of parametrized 
% martingale solution is {\it strictly} 
% weaker than the one of differential martingale solution. \EEE

We can now state the main result of this paper, which concerns
the existence of parametrized martingale solutions
for problem \eqref{eq0_bis} and 
reads as
follow. 
\begin{thm}[Existence of  parametrized  martingale solutions]
  \label{thm1}
   There exists a  parametrized  martingale solution 
  $(\Omega, \hat\cF,(\hat\cF_\tau)_{\tau\geq0},\P,\hat M, 
  \hat u, \hat u^d, \hat t, \hat v)$ of problem \eqref{eq0_bis}
  which additionally satisfies, for every $\ell\geq2$ and $\hat T>0$,
  \begin{align}
  &\hat u \in L^{p\ell}_{\hat\cP}(\Omega; C^0([0,\hat T]; H))\cap 
  L^{p\ell}_{w}(\Omega; L^\infty(0,\hat T; V))\,,\\
  &\hat v, B(\hat u) \in 
  L^\ell_{\hat\cP}(\Omega; L^1(0,+\infty;H))\cap
  L^{q\ell}_{\hat\cP}(\Omega; L^\infty(0,+\infty; V^*))\,.
  \end{align}
  Moreover, the following energy identity holds
  for every $\tau\geq0$:
  \begin{align}
  \label{en}
  \E\Phi( \hat u(\tau)) 
  + \E\int_0^\tau\left( \hat v(\sigma), \partial_t \hat u^d(\sigma)\right)_H\,\d \sigma
  =\Phi(u_0) 
  +\frac12\E\int_0^\tau\operatorname{Tr}\left[L(\hat t(\sigma), \hat u(\sigma))\right]
  \hat t'(\sigma)\,\d \sigma\,.
\end{align}
\end{thm}

This existence result in Section \ref{sec:limit} by means of a
viscous-regularization method.
\UUU Note that the role of the exponent $\ell$ here is to 
specify the order of the moments of the stochastic processes in play, 
namely their integrability in $\Omega$.
\EEE

\UUU
\section{Comparison with the deterministic setting}
\label{sec:deterministic}

Before moving on with the proof of our main results Theorems
\ref{thm0}-\ref{thm1}, let us provide in this section some additional discussion
on the applicative relevance of the rate-independent SPDE \eqref{eq0},
as well as some detail on  our variational setting. Notation and assumptions are as in Section \ref{sec:main}.

To start with, let us step back and consider the classical {\it
  deterministic} situation, where no stochasticity is assumed. In this case, the unknown trajectory $v$ of a general
rate-independent system is a function of
time only, namely $v:[0,T]\to V$. In particular, we assume that $v(t)$
identifies the state of a rate-independent system at time $t\in
[0,T]$, seen as a point in the abstract Hilbert space $V$. The
evolution of the rate-independent system is modeled as the
outcome of a balance of energy-storage and dissipation mechanisms, under
the influence of external actions. More precisely, 
the evolution is assumed to be driven by
the nonlinear deterministic relation
\begin{equation}
  \label{eq:comp1}
A(\partial_t v) +   \partial \Phi (v)= g\quad \text{in $V^*$, a.e. in} \ (0,T)
\end{equation}
starting from some given initial state, namely, $v(0)=v_0$. Recall that the symbol
$\partial_t v$ indicates the time derivative of $v$. The potential  
  $\Phi: V\to
 (-\infty,\infty]$ represents the stored energy of the systems, so
 that the subdifferential  $\partial \Phi (v) $ is the system of
 conservative actions at state $v$ (in this
section, we assume $\partial \Phi$ to be single-valued, so to simplify
notation). On the other hand, $A(\partial_t v)$
 models the system of dissipative actions corresponding to the rate
 $\partial_t v$. Eventually, $g : [0,T]   \to H$ 
 corresponds to some given time-dependent external actions.

 The rate-independence of the system follows by   
 choosing $A$ to be positively $0$-homogeneous, which is indeed the
 case of \eqref{A}. 
  Indeed, by letting $t \mapsto v(t)$ solve \eqref{eq:comp1}, for
all increasing diffeomorphisms $\phi: [0,\hat T) \to [0,T)$ with
$\phi(0)=0$ one has that $v \circ\phi$ solves \eqref{eq:comp1}
with $g$ replaced by $g \circ\phi$.

The balance relation \eqref{eq:comp1} is the departing point for the
general theory of rate-independent systems \cite[Ch.~1,
p.~2]{Mielke-Roubicek15}. Under different different choices of $V$, $H$,
$A$, and $\Phi$, such problem can be obtained as variational
formulation in different settings. A first examples is the scalar {\it play operator} \cite{Visintin94}
$$V=H=\Rz, \quad A(v):=
\left\{
  \begin{array}{ll}
    \disp v/|v|&\ \ \text{if} \ v \not =0 \\[1mm]
    {[-1,1]}&\ \ \text{if} \ v  =0 
  \end{array}
  \right., \quad \Phi(v) = \frac12 v^2$$
acting indeed as building block of more complex hysteretic operators
\cite{Brokate-Sprekels96,Krejci96}. 
Its infinite-dimensional generalization
$$V=H, \quad A(v):=
\left\{
  \begin{array}{ll}
    \disp v/\| v \|_H &\ \ \text{if} \ v \not =0 \\[1mm]
                    \{w \in H \: : \:
\| w \|_H\leq 1  \}&\ \ \text{if} \ v  =0      
  \end{array}
  \right., \quad \Phi (v) = \frac12
\|v\|^2_H $$
is generally referred to as {\it sweeping
  process} \cite{Moreau77}. In this case, \eqref{eq:comp1} is usually
written in form of differential inclusion as
\begin{equation}
  \label{eq:comp11}
\partial_t v \ni N(g-v) \quad \text{in} \ H, \ \text{a.e. in} \ (0,T) 
\end{equation}
where $N(y)$ indicates the normal cone to the unit ball $\{\| x \|_H
\leq 1\}$ at the point $y$.
This rate-independent dynamics arises in connection with a 
arise in connection with different mechanical problems \cite{Monteiro}, including the
flow of a liquid in a cavity, unilateral contact, and plasticity,
possibly in some simplified or finite dimensional setting \cite{MoreauA}.

% A further classical example is that of {\it linearized
%   elastoplasticity}
% \begin{align}
%   &V = H^1_0({\mathcal O};\Rz^3)\times L^2(\Omega;\Rz^{3\times 3}),
%     \quad H = L^2({\mathcal O};\Rz^3)\times L^2(\Omega;\Rz^{3\times 3}),\\
%    & A(\partial_tv) = R\partial |\partial_tp|, \quad \Phi(v) = \frac12 \int_{\mathcal O} (\epsi(v)-p): {\mathbb C} (\epsi(v)-p) \, \d
%     x .\label{eq:plast}
% \end{align}
% Here, ${\mathcal O} \subset \Rz^3 $ is the reference configuration of
% an elastoplastic body (supposed to be nonempty, open, connected,
% bounded and smooth), $v(t): {\mathcal O}\to \Rz^3$ indicates its
% displacement, and $p(t):  {\mathcal O}\to \Rz^{3\times 3}$ is its
% plastic strain. The positive-definite 4-tensor $ {\mathbb C} $ is the
% elasticity tensor and $\epsi(v) = (\nabla v + \nabla v^\top)/2$ is the
% symmetrized strain. We imposed homogeneous Dirichlet conditions to
% $v$, for simplicity. Eventually, the parameter $R>0$ is the yield
% stress of the plastic transformation and $g(t):  {\mathcal O}\to
% \Rz^3$ models a time-dependent body force \cite{HanReddy}. 

The operator $A$ itself is often the
subdifferential of a {\it dissipation} potential $D: H \to
[0,+\infty]$, which is assumed to be positively $1$-homogeneous and
convex with $D(0)=0$. Assumption \eqref{A} corresponds indeed to the
choice $D(v) = \| v \|_H$. Note however that the
existence of such a dissipation potential, although common in
applications, is not strictly necessary. Indeed, what is actually needed is that the  instantaneous dissipation
$(A(\partial_t v), \partial_t v)\geq 0$ is nonnegative, which would 
follow  for $A $ monotone with $0\in
A(0)$.

By testing relation \eqref{eq:comp1} on $\partial_t v$ and assuming
sufficient smoothness, one obtains the deterministic energy identity
\begin{align}
 & \Phi(v(\tau)) -(g(\tau),v(\tau))+ \int_0^\tau (A(\partial_t v(\sigma)), \partial_t
   v(\sigma)) \,\d \sigma \nonumber\\
  &\qquad = \Phi(v_0)  -(g(0),v_0) - \int_0^\tau
  (\partial_t g(\sigma), v(\sigma))\, \d \sigma \quad \forall \tau \in
    [0,T]. \label{eq:comp2}
\end{align}
This balance expresses the fact that the discrepancy in complementary
energy $\Phi(v)-(g,v)$ between times $0$ and $\tau$ is balanced by the
dissipated energy  
$\int_0^\tau (A(\partial_t v),\partial_tv)$ and by the work of
external actions $-\int_0^\tau
(\partial_t g , v)$. The proof of this deterministic energy identity hinges on the validity of a chain rule
for the potential $\Phi$.

In many applications, uncertainty and randomness are significant
and cannot be neglected. By incorporating randomness in the
modelization, one aims at delivering a more realistic modelization of
complex systems, combining deterministic dynamics and stochastic
fluctuations. In the setting of rate-independent or hysteretic
systems, the onset of stochastic effects is well-documented. Among
many others, the readers is referred
\cite{Belko,Bertotti,Das,deJager,Pajaro,Sides} for different examples,
ranging from the nanoscale of material science, to genetics, ecology,
to the macroscale of climate.
The mathematical treatment of stochastic differential systems with
hysteresis has also been  developed, see \cite{Bensoussan,Berglund,Berglund2}. One may specifically to mention the body of work revolving around the {\it
       Skorohod problem} \cite{Bernicot, ElKaroui,Lions,Tanaka}. This corresponds to a stochastic version of the sweeping process
     \eqref{eq:comp11} and is discussed in more detail below.

We include stochasticity in the problem in terms of an additional 
external source of stochastic type. In particular, the deterministic
forcing $g(t)\, \d t$ is augmented by a stochastic forcing of
multiplicative type, i.e., possibly also depending on the
state. Namely, we assume that the system \eqref{eq:comp1} is driven by
the forcing
$$g(t)\, \d t  + G(t,u(t))\, \d W$$
for some given $G:[0,T]\times H \to H$. Here, $W$ represents a
cylindrical Wiener process, the choice of this particular noise being
merely motivated by simplicity.

As the system is now stochastic, in addition to space and time, the
state $u$ will also depend on the specific realization in the
filtered probability space $(\Omega,\cF,(\cF_t)_{t\geq0},\P)$. As realizations
of the cylindrical Wiener process are almost surely continuous but not differentiable in
time, the presence of
the stochastic
forcing prevents trajectories $t \mapsto u(t)$ from being smooth. We are hence forced to handle trajectories by decomposing
\begin{equation}
  \d u(t) = \partial_t u^d(t) \, \d t + u^s(t) \, \d
W(t),\label{eq:comp3}
\end{equation}
see \eqref{ito_process}. Here, $\partial_t u^d \, \d t $ and $u^s \,
\d W$ represent the absolutely continuous and singular parts (with respect to the Lebesgue measure on
$[0,T]$) of   the measure $\d u$, respectively. Note that the fact that the singular part of
$\d u$ is absolutely continuous with respect to $\d W$ is inspired by
the theory of stochastic ordinary differential equations and is part
of our modeling assumptions.  

The factorization \eqref{eq:comp3} plays a role in specifying the
stochastic analogue to \eqref{eq:comp1}. In particular, the
deterministic dissipative
term $A(\partial_tv)$ needs to be modified for trajectories which are
not differentiable with respect to time. In the stochastic case, we
consider the augmented   dissipation term
\begin{equation}
  \label{eq:comp4}
  A(\partial_t u^d)\, \d t + u^s \, \d W.
\end{equation}
This decomposition is our main modeling assumption.  The absolutely continuous term still provides a
$0$-homogeneous contribution. On the other hand, we assume that the stochastic part
contributes linearly to the dissipation. We argue that this linear
choice is not restrictive. Note at first that it is
consistent with the deterministic one when the noise is switched off (i.e., for $G=0$). Secondly, rate-independence is
guaranteed. Indeed, the stochastic version of \eqref{eq:comp1} reads
\begin{equation}
  \label{eq:comp5}
A(\partial_t u^d)\, \d t +u^s \, \d W+   B (u)\, \d t \ni g \, \d t +
G(\cdot,u)\, \d  W \quad \text{in $V^*$, a.e. in} \ (0,T)
\end{equation}
with initial condition $u(0)=u_0$. 
By letting $t \mapsto u(t)$ solve \eqref{eq:comp5}, for
all increasing diffeomorphism $\phi: [0,\hat T) \to [0,T)$ with
$\phi(0)=0$ one has that $\hat u(t) = u(\phi(t))$ solves \eqref{eq0}
with $G(\cdot,u)$ replaced by $G(\phi(\cdot),u)$
 and $W$ replaced by $W \circ \phi$, which is a cylindrical continuous martingale. Thirdly, as problem
 \eqref{eq:comp5} decomposes in the system
 \begin{align}
   &A(\partial_t u^d) +   B (u)  \ni g  \quad \text{in $V^*$, a.e. in} \
     (0,T),\label{eq:comp51}\\
   &u^s=
     G(\cdot,u) \quad \text{in $\cL^2(U,H)$, a.e. in} \ (0,T),\label{eq:comp52}
     \end{align}
 alternatively to \eqref{eq:comp4} one could consider the more general form $A(\partial_t u^d)\, \d t +
     S(u^s) \, \d W$, under suitable
     invertibility and boundedness requirements on the operator $S:
     \cL^2(U,H) \to \cL^2(U,H)$, upon changing $G$ accordingly. In this sense, choosing the
     dissipative term to be linear in $u^s$ is no actual restriction.

     Before moving on, let us further comment on the decomposition
     \eqref{eq:comp3} in light of the above-mentioned 
       Skorohod problem. This is a stochastic version of the sweeping process \eqref{eq:comp11}
and reads
\cite{Bernicot,Tanaka}
\begin{equation}
  \label{eq:comp111}
\d u  \in  N(g-u) \, \d t + G\, \d W \quad \text{in} \ H, \
\text{a.e. in} \ (0,T) 
\end{equation}
(note that $g=0$ is often considered \cite{Bernicot,ElKaroui,Lions}). 
By assuming position \eqref{eq:comp52} and $V=H$ one is allowed to rewrite
  \eqref{eq:comp5} as the following doubly nonlinear variant of the 
  Skorohod problem 
  \begin{equation}
  \label{eq:comp112}
\d u  \in  N(g-\partial \Phi(u)) \, \d t + G\, \d W \quad \text{in} \ H, \ \text{a.e. in} \ (0,T). 
\end{equation}
Indeed, using $u^s = G(\cdot,u)$ we have that inclusion \eqref{eq:comp112}
implies 
$$   \partial_t u^d  \in N(g-\partial \Phi(u)) \quad \text{in} \ H, \
\text{a.e. in} \ (0,T) $$
and \eqref{eq:comp51} follows  
by inverting the maximal monotone map $N$, since $N^{-1}=A$.

 To further compare the stochastic case of \eqref{eq:comp5} to its
 deterministic counterpart \eqref{eq:comp4}, one may compute the energy
 balance at the stochastic level. The required chain-rule computation
 for $\Phi$ calls now for the application of an It\^o formula. By
 assuming smoothness we get
\begin{align}
 & \Phi(u(\tau)) -(g(\tau),u(\tau))+ \int_0^\tau (A(\partial_t u^d(\sigma)), \partial_t
   u^d(\sigma)) \,\d \sigma \nonumber\\
  &\qquad = \Phi(u_0)  -(g(0),u_0) - \int_0^\tau
    (\partial_t g(\sigma), u(\sigma))\, \d \sigma \nonumber\\
  &\qquad + \int_0^t {\rm Tr}\,L(\sigma,u(\sigma))  \d \sigma  +
    \int_0^t(\partial \Phi(u(\sigma)),
    G(\sigma,u(\sigma))\, \d W(\sigma))_H
    \quad \P-\text{a.s.},
    \ \  \forall \tau \in
    [0,T]\nonumber
\end{align}
where we use the shorthand notation $L(\sigma,u(\sigma))= G(\sigma,u(\sigma))
    G(\sigma,u(\sigma))^*D_{\mathcal G}\partial \Phi (u(\sigma))$. Note that $u^s$
    does not appear in the latter, as effect of position \eqref{eq:comp52},
    i.e., of the factorization \eqref{eq:comp4}. As $\partial \Phi$ is
    monotone, the operator $L$ is nonnegative. 
Compared with the deterministic energy balance \eqref{eq:comp2}, we
have additional two terms of stochastic origin. By taking
expectations, the last term involving $G$ disappears and we
are left with
\begin{align}
 & \E\Phi(u(\tau)) -\E (g(\tau), u(\tau))+ \E\int_0^\tau (A(\partial_t u^d(\sigma)), \partial_t
   u^d(\sigma)) \,\d \sigma \nonumber\\
  &\qquad = \E\Phi(u_0)  -\E (g(0),u_0) - \E\int_0^\tau
    (\partial_t g(\sigma), u(\sigma))\, \d \sigma \nonumber\\
  &\qquad  + \E  \int_0^t {\rm Tr}\,L(\sigma,u(\sigma))  \d \sigma 
    \quad \P-\text{a.s.},
    \ \  \forall \tau \in
    [0,T]. \label{eq:comp6}
\end{align}
Hence, noise contributes and additional nonnegative term to
the expected energy
balance. In the context of stochastic integration, this is not
surprising. From the modeling viewpoint, the persistence of stochastic
effects in the deterministic energy balance \eqref{eq:comp6} is the signature of
the stochastic nature of the forcing. In the rest of the paper, we let $g=0$ for brevity. Extending
our results to the case $g\not =0$ would be straightforward.

Before closing this section, let us remark that our analysis of
\eqref{eq:comp5} is currently limited to Hilbert spaces, see \cite{EM,MRS12},
as well. From the applicative
viewpoint, one would be interested in considering less regular
settings. Indeed, various applications feature $L^1$-type
dissipations $D$, suggesting to set the problem in 
 a nonreflexive Banach space. This would require an integration theory
 in such spaces,  necessarily
resulting in additional technical intricacies and presently being out of the
reach of our analysis.

\EEE

\section{Differential vs parametrized martingale solutions: Proof
of Theorem \ref{thm0}}
\label{sec:sol}

We devote this section to the comparison of  differential 
 and parametrized martingale solutions. This in particular brings to the proof
of Theorem~\ref{thm0}.

Let $(\Omega, \cF,(\cF_t)_{t\in[0,T]},\P, W, u, u^d, v)$ be a
 differential  martingale 
solution to problem \eqref{eq0_bis} in the sense of Definition~\ref{def:vs}.
We define the random variable 
\[
  \hat T:\Omega\to\erre_+\,, \qquad
  \hat T(\omega):=T + \norm{\partial_tu^d(\omega)}_{L^1(0,T; H)}\,, \quad\omega\in\Omega\,,
\]
and the stochastic process
\[
  \tau_0:\Omega\times[0,T]\to\erre_+\,, \qquad
  \tau_0(t):=t+\int_0^t\norm{\partial_t u^d(s)}_H\,\d s\,, \quad s\in[0,T]\,.
\]
By the regularity of $u^d$, there exists $\Omega_0\in\cF$ with $\P(\Omega_0)=1$
such that $\tau_0(\omega)\in C^0([0,T])$ for all $\omega\in\Omega_0$. 
Consequently,
for all $\omega\in\Omega_0$ the trajectory 
\[
  \tau_0(\omega):[0,T]\to[0,\hat T(\omega)]
\]
is continuous increasing, hence also invertible with continuous increasing inverse 
\[
  \tau_0(\omega)^{-1}:[0,\hat T(\omega)]\to[0,T]\,.
\]
We define the {\it rescaling} process $\hat t$ as
\[
  \hat t:\Omega\times\erre_+\to [0,T]\,,
  \qquad
  \hat t(\omega, \tau):=
  \begin{cases}
  \tau_0(\omega)^{-1}(\tau) \quad&\text{if } \tau\in[0,\hat T(\omega)]\,,
  \;\omega\in\Omega_0\,,\\
  T=\tau_0(\omega)^{-1}(\hat T(\omega))
  \quad&\text{if }\tau>\hat T(\omega)\,,\;\omega\in\Omega_0\,,\\
  0 \quad&\text{if } \omega\in\Omega\setminus\Omega_0\,.
  \end{cases}
\]
We collect the measurability properties of $\tau_0$ in the following lemma.
\begin{lem}
  \label{lem:t_0}
  $\hat t$ is a well--defined  
  progressively measurable continuous nondecreasing process. Moreover,
  $\hat t(\cdot,\tau):\Omega\to[0,T]$ is a bounded stopping time 
   for every $\tau\geq0$.
\end{lem}
\begin{proof}
The fact that $\hat t$ is continuous and non-decreasing
follows directly from its definition.
To prove progressive measurability we need to show that
for every $\bar \tau>0$ the restriction 
$\hat t:\Omega\times[0,\bar \tau]\to[0,T]$
is $\cF_{\bar \tau}\otimes\cB([0,\bar\tau])$--measurable,
i.e.~that given an arbitrary $\bar \tau>0$ and $\bar t\in[0,T]$
it holds 
\[
  \left\{(\omega,\tau)\in\Omega\times[0,\bar\tau]:\;
  \hat t(\omega,\tau)<\bar t\right\} \in 
  \cF_{\bar \tau}\otimes\cB([0,\bar\tau])\,.
\]
In fact, it suffices to argue by replacing $\Omega$ with $\Omega_0$ above, for
$\Omega \setminus \Omega_0 \times [0,\bar\tau]\in 
  \cF_{\bar \tau}\otimes\cB([0,\bar\tau])$. 
Recall that 
by definition of $\tau$ we know that 
$\tau(s)\geq s$ for every $s\in[0,T]$ and,
consequently, also that 
\[
\hat t(\tau)\leq\tau\leq \bar \tau \qquad\forall\,\tau\in[0,\bar\tau]\,.
\]
Hence, if $\bar t> \bar \tau$ then we immediately have
\[
  \left\{(\omega,\tau)\in\Omega_0\times[0,\bar\tau]:\;
  \hat t(\omega,\tau)<\bar t\right\}=\Omega_0\times[0,\bar\tau] \in 
  \cF_{\bar \tau}\otimes\cB([0,\bar\tau])\,.
\]
Alternatively, if $\bar t\in[0,\bar \tau]$
we have, by the monotonicity of $\tau_0$,
the fact that $\bar t\in[0,T]$, and
the bound $\tau_0\leq \hat T$,
\begin{align*}
  &\left\{(\omega,\tau)\in\Omega_0\times[0,\bar \tau]:\;\hat t(\omega,\tau)<\bar t\right\}\\
  &=\left\{(\omega,\tau)\in\Omega_0\times[0,\bar \tau]
  :\;\hat t(\omega,\tau)<\bar t\,,\;
  \tau\leq\hat T(\omega)\right\} \\
  &\qquad\cup
  \left\{(\omega,\tau)\in\Omega_0\times[0,\bar \tau]:
  \;\hat t(\omega,\tau)<\bar t\,,\;
  \tau>\hat T(\omega)\right\}\\
  &=
  \left\{(\omega,\tau)\in\Omega_0\times[0,\bar \tau]
  :\;\tau_0(\omega)^{-1}(\tau)<\bar t\,,\;
  \tau\leq\hat T(\omega)\right\}\\
  &\qquad\cup
  \left\{(\omega,\tau)\in\Omega_0\times[0,\bar \tau]
  :\;T< \bar t\,,\;
  \tau>\hat T(\omega)\right\}\\
  &=
  \left\{(\omega,\tau)\in\Omega_0\times[0,\bar \tau]
  :\;\tau<\tau_0(\omega,\bar t)\wedge\hat T(\omega)\right\}\\
  &=
  \left\{(\omega,\tau)\in\Omega_0\times[0,\bar \tau]
  :\;\tau<\tau_0(\omega,\bar t)\right\}\,.
\end{align*}
Since $\tau_0(\cdot,\bar t)$ is $\cF_{\bar t}$--measurable and $\bar t\leq\bar\tau$ this yields 
\[
  \left\{(\omega,\tau)\in\Omega_0\times[0,\bar \tau]:\;\hat t(\omega,\tau)<\bar t\right\}
  \in \cF_{\bar t}\otimes \cB([0,\tau])\subseteq\cF_{\bar \tau}\otimes \cB([0,\tau])\,.
\]
This shows that the stochastic process $\hat t$ is progressively measurable.
Furthermore, with a similar argument 
it is possible to show that $\hat t(\cdot, \tau)$ is a stopping time
for every $\tau\geq0$. Indeed, let arbitrary $\tau\geq0$ and $\bar t\in[0,T]$:
if $\bar t>\tau$ we have 
\[
  \left\{\omega\in\Omega_0:\;\hat t(\omega,\tau)<\bar t\right\}=\Omega_0\in\cF_\tau\,,
\]
while if $\bar t\leq \tau$ we get, as before,
\[
  \left\{\omega\in\Omega_0:\;\hat t(\omega,\tau)<\bar t\right\}=
  \left\{\omega\in\Omega_0:\;\tau<\tau_0(\omega,\bar t)\right\}\in\cF_{\bar t}\,.
\]
Recalling that the filtration $(\cF_t)_{t\geq0}$ is assumed to be right--continuous,
this shows indeed that $\hat t(\cdot,\tau)$ is a stopping time for every $\tau\geq0$.
\end{proof}

Now, we define
\begin{align*}
  \hat u: \Omega\times\erre_+\to V\,,& \qquad
  \hat u(\omega,\tau):=
  u(\omega,\hat t(\omega,\tau))\,, &&(\omega,\tau)\in\Omega\times\erre_+\,,\\
  \hat u^d: \Omega\times\erre_+\to H\,,& \qquad
  \hat u^d(\omega,\tau):=
  u^d(\omega,\hat t(\omega,\tau))\,, &&(\omega,\tau)\in\Omega\times\erre_+\,,\\
  \hat G:\Omega\times\erre_+\times H\to\cL^2(U,H)\,,&
  \qquad
  \hat G(\omega,\tau,z):=G(\hat t(\omega,\tau), z)\,, 
  &&(\omega,\tau,z)\in\Omega\times\erre_+\times H\,,\\
  \hat M: \Omega\times\erre_+\to U_1\,,& \qquad
  \hat M(\omega,\tau):=
  W(\omega,\hat t(\omega,\tau))\,, &&(\omega,\tau)\in\Omega\times\erre_+\,,\\
  \hat v:\Omega\times\erre_+\to H\,,&
  \qquad
  \hat v(\omega,\tau):=v(\omega,\hat t(\omega,\tau))\,, 
  &&(\omega,\tau)\in\Omega\times\erre_+\,.
\end{align*}
As far as the rescaled filtration is concerned, we note that 
by Lemma~\ref{lem:t_0} it is well-defined
\[
  \hat\cF_{\tau}:=\cF_{\hat t(\cdot, \tau)}=
  \left\{F\in\cF:\; F\cap\{\omega\in\Omega:\hat t(\omega,\tau)\leq t\}
  \in\cF_t\quad\forall\,t\in[0,T]\right\}\,, \qquad
  \hat\cF:=\cF\,.
\]

\begin{lem}
 \label{lem:hat_0}
 The following holds.
 \begin{enumerate}[(i)]
 \item\label{lem:hat1_0} $(\hat\cF_{\tau})_{\tau\geq0}$
 is a filtration satisfying the usual conditions, with associated 
 progressive $\sigma$-algebra denoted by $\hat \cP$.
 \item\label{lem:hat2_0} 
 $\hat t$, $\hat u$, $\hat u^d$, and $\hat v$ are progressively measurable
 on $(\Omega,\hat \cF,(\hat\cF_{\tau})_{\tau\geq0},\P)$, and satisfy 
 \begin{align*}
 \hat t\in W^{1,\infty}_{loc}(0,+\infty) \qquad&\P\text{-a.s.}\,,\\
 \hat u\in C^0([0,+\infty); H)\cap L^\infty(0,+\infty; V) \qquad&\P\text{-a.s.}\,,\\
 \hat u^d\in W^{1,\infty}_{loc}(0,+\infty; H)\,, \qquad&\P\text{-a.s.}\,,\\
 \partial_t \hat u^d = \partial_t u^d(\hat t)\hat t'\,,
 \qquad&\P\text{-a.s.}\,,\\
 \hat t'(\tau) + \norm{\partial_t\hat u^d(\tau)}_H=1
 \qquad&\text{for a.e.~}\tau>0\,,\quad\P\text{-a.s.}\,,\\
 \hat v\in L^1_{loc}(0,+\infty; H)\,, \qquad&\P\text{-a.s.}\,,
 \end{align*}
 \item\label{lem:hat3_0} $\hat M$ is a continuous square-integrable $U$-cylindrical martingale 
 on $(\Omega,\hat \cF,(\hat\cF_{\tau})_{\tau\geq0},\P)$, with 
 tensor quadratic variation given by 
 \[
 \qqv{\hat M}(\tau)=Q_1\hat t(\tau)\,, \qquad\tau\geq0\,.
 \]
 \item\label{lem:hat4_0} For every $Y\in L^2_\cP(\Omega; L^2(0,T; \cL^2(U,H)))$,
 the rescaled process $\hat Y:=Y\circ\hat t$ is 
 stochastically integrable with respect to $\hat M$ 
 on $(\Omega,\hat \cF,(\hat\cF_{\tau})_{\tau\geq0},\P)$ and
 \[
 \int_0^{\hat t(\tau)}Y(s)\,\d W(s)=
 \int_0^\tau \hat Y(\sigma)\,\d\hat M(\sigma) \qquad
 \forall\,\tau\geq0\,,\quad\P\text{-a.s.}
 \]
 In particular,
 $\hat G(\cdot, \hat u)$ is 
 stochastically integrable with respect to $\hat M$ 
 and it holds that
 \[
 \int_0^{\hat t(\tau)}G(s,u(s))\,\d W(s)=
 \int_0^\tau \hat G(\sigma, \hat u(\sigma))\,\d\hat M(\sigma) \qquad
 \forall\,\tau\geq0\,,\quad\P\text{-a.s.}
 \]
 \end{enumerate} 
\end{lem}
\begin{proof}
   Ad~\ref{lem:hat1}. Let $\sigma,\tau\geq0$, $\sigma\leq\tau$,  be arbitrary:
  then, for every $F\in\hat\cF_{\sigma}$ and for every $t\in[0,T]$,
  since $\hat t(\cdot,\sigma)\leq\hat t(\cdot,\tau)$ almost surely,
   it holds that 
  \[
  F\cap\{\omega\in\Omega:\hat t(\omega,\tau)\leq t\}=
  F\cap\{\omega\in\Omega:\hat t(\omega,\sigma)\leq t\}\cap
  \{\omega\in\Omega:\hat t(\omega,\tau)\leq t\}\,,
  \]
  where $F\cap\{\omega\in\Omega:\hat t(\omega,\sigma)\leq t\}\in\cF_t$
  because $F\in\hat\cF_{\sigma}$ and 
  $\{\omega\in\Omega:\hat t(\omega,\tau)\leq t\}\in\cF_t$
  \UUU since \EEE $\hat t(\cdot,t)$ is a stopping time by Lemma~\ref{lem:t_0}.
  This implies that $F\in\hat\cF_{\tau}$, 
  so that $\hat\cF_{\sigma}\subseteq\hat\cF_{\tau}$:
  by arbitrariness of $\sigma$ and $\tau$, it follows that
   $(\hat\cF_{\tau})_{\tau\geq0}$ is a filtration.
   Moreover, $\hat \cF_{0}=\cF_0$ by definition, hence $\hat \cF_{0}$
   is trivially saturated. Also, let $\tau\geq0$ be arbitrary, 
   and $F\in\hat\cF_{\tau+\delta}$ for every $\delta>0$:
   then, for every $\lambda>0$ fixed and $t\in[0,T]$, one has
   \[
   F\cap \{\omega\in\Omega:\hat t(\omega,\tau)\leq t\}=
   F\cap\bigcap_{\delta\in(0,\lambda)}\{\omega\in\Omega:\hat t(\omega,\tau)\leq t+\delta\}
   \in\cF_{t+\delta}\subseteq\cF_{t+\lambda}\,.
   \]
   By arbitrariness of $\lambda>0$ and the fact that 
   $(\cF_t)_{t\in[0,T]}$ is right-continuous, we get
   \[
   F\cap \{\omega\in\Omega:\hat t(\omega,\tau)\leq t\}\in
   \bigcap_{\lambda>0}\cF_{t+\lambda}=\cF_t\,,
  \]
  from which it follows that $F\in\hat\cF_{\tau}$ by arbitrariness of $t$. Hence,
  also $(\hat\cF_{\tau})_{\tau}$ is right-continuous.\\
   Ad~\ref{lem:hat2}. 
  The $\hat \cP$-measurability of $\hat t$
  follows from the fact that $\hat t$ has continuous 
  trajectories and is adapted to $(\hat\cF_{\tau})_{\tau\geq0}$
  by definition.
  The $\hat\cP$-measurability of 
  $\hat u$, $\hat u^d$, and $\hat v$ follows from the 
  $\cP$-measurability of $u$, $u^d$, and $v$, 
  and the fact that 
  the function 
  $(\omega,t)\mapsto(\omega,\hat t(\omega,t))$
  is $\hat\cP/\cP$-measurable. 
  Furthermore, by definition of $\hat t$ it holds that 
 \[
  \sup_{\tau\geq0}\norm{\hat u(\omega,\tau)}_V=
  \sup_{t\in[0,T]}\norm{u(t)}_V \qquad\P\text{-a.s.}\,,
  \]
  so that
  the regularity of $\hat u$ 
  immediately follows from the 
  continuity of the trajectories of $\hat t$.
  As for the regularity of $\hat u^d$,
  we have that $\hat t$ is almost everywhere differentiable 
  with 
  \[
  \hat t'(\tau) =
  \left(1 + \norm{\partial_t u^d(\hat t(\tau))}_H\right)^{-1} 
  \qquad
  \text{for a.e.~}\tau>0\,,\quad\P\text{-a.s.}\,,
\]
from which 
\[
  \hat t'(\tau) + \norm{\partial_t u^d(\hat t(\tau))}_H\hat t'(\tau)=1
  \qquad
  \text{for a.e.~}\tau>0\,,\quad\P\text{-a.s.}
\]
Hence, $\partial_t u^d(\hat t)\hat t'\in L^\infty(0,+\infty; H)$
and integration by
substitution yields 
\[
  \hat u^d(\tau)=u^d(\hat t(\tau))=
  u_0 + \int_0^{\hat t(\tau)}\partial_t u^d(s)\,\d s
  =u_0 + \int_0^\tau\partial_t u^d(\hat t(\sigma))\hat t'(\sigma)\,\d\sigma\,,
\]
from which $\hat u^d\in W^{1,\infty}_{loc}(0,T; H)$ with 
$\partial_t\hat u^d=\partial_tu^d(\hat t)\hat t'$.
Also, for all $\tau>0$ we have, integrating by substitution, that 
\begin{align*}
  \int_0^{\tau}\norm{\hat v(\sigma)}_H\,\d\sigma&=
  \int_0^{\hat t(\tau)}\norm{v(s)}_H\tau'(s)\,\d s
  \leq\int_0^T\norm{v(s)}_H\norm{\partial_t u^d(s)}_H\,\d s\,,
\end{align*}
so that the regularity of $\hat v$ follows from the ones of $v$ and
$u^d$.\\
Ad~\ref{lem:hat3}.
Again, $\hat M$ is $\hat\cP$--measurable by definition of the 
new filtration $(\hat\cF_{\tau})_{\tau\geq0}$ and  by  the fact that 
$W$ is $\cP$--measurable. Moreover, the continuity of the trajectories 
of $\hat M$ in $U_1$ is a direct consequence of the continuity of the trajectories 
of $W$ in $U_1$ and of $\hat t$. The fact that $\hat M$ is 
a martingale with values in $U_1$ is a consequence of the Optional Stopping Theorem
(see \cite[Thm.~13.5]{metivier})
and the fact that $W$ is a martingale. Indeed, for any arbitrary $\sigma,\tau\geq0$, with $\sigma\leq\tau$,
$\hat t(\cdot,\sigma)$ and $\hat t(\cdot,\tau)$ are bounded stopping times
satisfying $\hat t(\cdot,\sigma)\leq \hat t(\cdot,\tau)$ almost surely, from which 
\[
  \E\Big[\hat M(\tau)\Big|\hat\cF_\sigma\Big]=
  \E\Big[W(\hat t(\cdot,\tau))\Big|\cF_{\hat t(\cdot,\sigma)}\Big]=
  W(\hat t(\cdot,\sigma))=\hat M(\sigma)\,.
\]
Furthermore, by definition of $W$ we have that 
\[
  N(t):=(W\otimes W)(t) - Q_1t\,, \qquad t\in[0,T]\,,
\]
is a continuous $(\cF_t)_{t\in[0,T]}$--martingale 
with values in $\cL^1(U_1,U_1)$. Using the same arguments 
as above, one readily sees that setting
\[
  \hat N(\tau):=N(\hat t(\tau))=
  (\hat M\otimes \hat M)(\tau) - Q_1\hat t(\tau)\,, \qquad \tau\geq0\,,
\]
it holds that $\hat N$ is a continuous martingale with values in $\cL^1(U_1,U_1)$
with respect to the filtration $(\hat\cF_{\tau})_{\tau\geq0}$: it follows then that 
the tensor quadratic variation of $\hat M$ is exactly $Q_1\hat t$.\\
Ad~\ref{lem:hat4}.
Clearly, we have that $\hat Y$ is $\hat\cP$-measurable.
Moreover, for every $\tau\geq0$, integration by substitution and
the regularity of $Y$ yield 
\begin{align*}
  &\E\int_0^\tau\|\hat Y(\sigma)\|_{\cL^2(U,H)}^2\,\d\hat t(\sigma)=
  \E\int_0^\tau\|\hat Y(\sigma)\|_{\cL^2(U,H)}^2
  \hat t'(\sigma)\,\d\sigma\\
  &\qquad=\E\int_0^{\hat t(\tau)}\|Y(s)\|_{\cL^2(U,H)}^2\,\d s
  \leq\norm{Y}^2_{L^2(\Omega;L^2(0,T; \cL^2(U,H)))}<+\infty\,.
\end{align*}
Recalling Subsection~\ref{ssec:cyl}, this shows that $\hat Y$
is stochastically integrable with respect to $\hat M$ on $[0,\tau]$
for every $\tau\geq0$. In order to show the change-of-variable formula 
for the respective stochastic integral, for any $N\in\enne$ let 
$\hat\pi_N:=\{\tau_0=0,\tau_1,\ldots,\tau_N=\tau\}$ be any arbitrary partition of $[0,\tau]$
such that $| \hat \pi_N| \to0$ if $N\to\infty$, 
let $\xi_k\in L^2(\Omega, \hat \cF_{\tau_k};\cL^2(U,H)\cap\cL(U_1,H))$,
$k=0,\ldots,N-1$, such that 
the elementary process
\[
  \hat E_N(\sigma):=
  \sum_{k=0}^{N-1}
  \xi_k\mathbbm{1}_{(\tau_k,\tau_{k+1}]}(\sigma)\,, \qquad \sigma\in[0,\tau]\,,
\]
satisfies
\[
  \lim_{N\to\infty}\E\int_0^\tau\norm{\hat E_N(\sigma) - 
  \hat Y(\sigma)}_{\cL^2(U,H)}^2\,\d\hat t(\sigma)=0\,.
\]
Then, recalling again Subsection~\ref{ssec:cyl}
and the definition of $\hat M$, it holds that 
\begin{align*}
  \int_0^\tau\hat Y(\sigma)\,\d \hat M(\sigma)&=
  \lim_{N\to\infty}\sum_{k=0}^{N-1}
  \xi_k
  (\hat M(\tau_{k+1}) - \hat M(\tau_k)) &&\text{in } L^2(\Omega)\\
  &=
  \lim_{N\to\infty}\sum_{k=0}^{N-1}
  \xi_k
  (W(\hat t(\tau_{k+1})) - W(\hat t(\tau_k))) &&\text{in } L^2(\Omega)\,.
\end{align*}
 Introducing the processes
\[
  E_N(s):=
  \sum_{k=0}^{N-1}
  \xi_k\mathbbm{1}_{(\hat t(\tau_k),\hat t(\tau_{k+1})]}(s)\,, \qquad s\in[0,\hat t(\tau)]\,,
\]
by It\^o's identity and integration by substitution we have
\begin{align*}
  &\E\norm{\sum_{k=0}^{N-1}
  \xi_k(W(\hat t(\tau_{k+1})) - W(\hat t(\tau_k))) - 
  \int_0^{\hat t(\tau)}Y(s)\,\d W(s)}_H^2\\
  &=\E\norm{ \int_0^{\hat t (\tau)}  
  (E_N(s) - Y(s))\,\d W(s)
  }_H^2\\
  &= \E\int_0^{\hat t (\tau)}  
  \norm{E_N(s)  -Y(s)}_{\cL^2(U,H)}^2\,\d s\\
  &=\E\int_0^\tau\norm{\hat E_N(\sigma) - \hat Y(\sigma
  )}_{\cL^2(U,H)}^2\,\d\hat t(\sigma) \to 0\,.
\end{align*}
This yields the desired change-of-variable formula, and we conclude.
\end{proof}

Let us now finally come to the proof of Theorem
\ref{thm0}. Assume $(\Omega, \hat\cF,(\hat\cF_\tau)_{\tau\geq0},\P,
\hat M, \hat u, \hat u^d, \hat v, \hat t)$ to be a differential
martingale solution.
From Lemma~\ref{lem:hat_0} it follows directly \UUU that \EEE
conditions \eqref{hat_t}--\eqref{eq1_var'}
of Definition~\ref{def:rs} hold. Moreover,   from
\eqref{eq2_var}--\eqref{incl} we find that for a.e. $\tau$ 
\[
   \hat v(\tau) + B(\hat u(\tau)) = 0\,, \qquad
  \hat v(\tau)\in A\left(\partial_t u^d(\hat t(\tau))\hat
    t'(\tau)\right) = A(\partial_t \hat u^d(\tau)) \subset \partial
  \Psi_{\norm{\cdot}} (\partial_t \hat u^d(\tau)) \,.
\]
This implies also \eqref{eq2_var'}--\eqref{incl'}
thanks to the definition of $\Psi_{\norm{\cdot}}$.
In particular, we have checked that 
$(\Omega, \hat\cF,(\hat\cF_\tau)_{\tau\geq0},\P, \hat M, \hat u, \hat u^d, \hat v, \hat t)$
is a  parametrized  martingale solution in the sense of Definition~\ref{def:rs}.

Conversely, for any 
parametrized  martingale solution 
$(\Omega, \hat\cF,(\hat\cF_t)_{t\in[0,T]},\P, \hat M, \hat u, \hat u^d, \hat v, \hat t)$
that satisfies $\hat t(\hat S)=T$ almost surely and
$\|\partial_t\hat u^d\|_H<1$
almost everywhere in $[\![0,\hat S]\!]$, it holds that 
$\hat t'>0$ almost everywhere in $[\![0,\hat S]\!]$.
Consequently, 
for almost every $\omega\in\Omega$
the trajectory $\hat t(\omega):[0,\hat S(\omega)]\to[0,T]$ is 
invertible, with inverse $\tau(\omega):[0,T]\to\erre_+$.  In
particular $A(\partial_t \hat u^d(\tau)) = \partial
  \Psi_{\norm{\cdot}} (\partial_t \hat u^d(\tau)) $ and
 the same argument can be replicated reversely 
to obtain a  differential  martingale solution
in the sense of Definition~\ref{def:vs}.

\section{Viscous regularization}
\label{sec:limit} 
In order to prove the existence of  parametrized  martingale solutions,
we firstly tackle a viscous approximation of the problem
in Subsection \ref{ssec:visc}. In Subsection \ref{ssec:visc} we
introduce a viscous regularisation, depending on the
parameter $\eps>0$. A priori estimates on viscous
solutions are then derived in Subsection \ref{sec:ue}. The passage to
the limit as $\eps\to0$, leading to  parametrized  martingale solutions and
the proof of Theorem \ref{thm1}, is detailed in Section \ref{sec:ri}.

A word of warning: henceforth the symbol $C$ stands for any
generic positive constant, possibly depending on data \UUU and  \EEE changing from
line to line, but independent of $\eps$. If needed, \UUU dependency
\EEE of
$C$ will be indicated with pedices.

\EEE

\subsection{Viscous relaxation}
\label{ssec:visc}
For every $\eps>0$, the viscously regularised dissipation potential is
\[
  \Psi_\eps:H\to[0,+\infty)\,, \qquad
  \Psi_\eps(z):=\frac\eps2\norm{z}_H^2 + \|z\|_H\,, \quad z\in H\,.
\]
The subdifferential of $\Psi_\eps$ can be classically computed as
\cite[Thm.~2.10]{barbu-monot}   $A_\eps=\partial\Psi_\eps$ given by
\[
  A_\eps:H\to 2^H\,, \qquad A_\eps(z):=\eps z + A(z)\,, \quad z\in H\,.
\]

We consider the viscously regularised problem 
\beq\label{eq_reg}
  \begin{cases}
  \d u_\eps = (\partial_t u_\eps^d)\,\d t + u_\eps^s\,\d W\,,\\
  u_\eps(0)=u_0\,,\\
  A_\eps(\partial_tu_\eps^d)  + B(u_\eps)
  \ni 0\,,\\
  u_\eps^s=G(\cdot,u_\eps)\,.
  \end{cases}
\eeq
For every $\eps>0$
the doubly nonlinear problem \eqref{eq_reg} admits a martingale solution 
thanks to the existence theory in \cite{ScarStef-SDNL2}.
More specifically, we have the following result.
\begin{prop}[Existence for a viscously regularized problem]
  \label{prop:WP_eps}
  There exists a filtered probability space 
  $(\Omega,\cF,(\cF_t)_{t\in[0,T]},\P)$ satisfying the usual conditions,
  a $U$-cylindrical Wiener process $W$ on it, 
  and a triple $(u_\eps,  u_\eps^d,  v_\eps)$, with
  \begin{align*}
  & u_\eps \in L^p_{  \cP}( \Omega; C^0([0,T]; H))\cap 
  L^p_{ w}( \Omega; L^\infty(0,T; V))\,,
  && u_\eps^d \in L^2_{  \cP}( \Omega; H^1(0,T; H))\,,\\
  & v_\eps \in L^2_{  \cP}( \Omega; L^2(0,T, H))\,,
  &&B( u_\eps) \in L^2_{  \cP}( \Omega; L^2(0,T; H))\,,
  \end{align*}
  such that 
  \begin{align}
  \label{eq1_var_eps}
  & u_\eps(t) = u_0 + \int_0^t\partial_t u^d_\eps(s)\,\d s 
  + \int_0^tG(s, u_\eps(s))\,\d   W(s)
  &&\forall\,t\in[0,T]\,,\quad  \P\text{-a.s.}\,,\\
  \label{eq2_var_eps}
  & v_\eps(t) + B( u_\eps(t)) = 0
  &&\text{for a.e.~$t\in(0,T)$}\,,\quad  \P\text{-a.s.}\,,\\
  \label{incl_eps}
  & v_\eps(t) \in A_\eps(\partial_t  u^d_\eps(t)) 
  &&\text{for a.e.~$t\in(0,T)$}\,,\quad  \P\text{-a.s.}
  \end{align}
  Furthermore, the following energy equality holds for every $t\in[0,T]$, $ \P$-almost surely:
\begin{align}
\nonumber
  &\Phi( u_\eps(t)) 
  + \int_0^t\left( v_\eps(s), \partial_t u^d_\eps(s)\right)_H\,\d s \\
  &=\Phi(u_0) 
  +\frac12\int_0^t\operatorname{Tr}\left[L(s, u_\eps(s))\right]\,\d s
  +\int_0^t\left(B( u_\eps(s)), G(s, u_\eps(s))\,\d  W(s)\right)_H\,.
  \label{en_eps}
\end{align}
\end{prop}
\begin{proof} The existence results of \cite{ScarStef-SDNL2} can be
  applied to the current setting, provided that we check that $A_\eps$
  is coercive on $H$ and linearly bounded. This in fact
  follows from 
  assumption {\bf A} since we have that 
  \begin{alignat*}{2}
  (w,z)_H =  \eps\norm{z}_H^2 +  \norm{z}_H 
  \ \  \text{and} \ \  \norm{w}_H\leq \eps \norm{z}_H +  1 
   \qquad\forall\,z\in H\,, \ \forall\,w\in A_\eps(z)\,.
 \end{alignat*}
 In fact, the results of
 \cite{ScarStef-SDNL2} apply to the case of a maximal $p$--strongly
 monotone $B$ on $V$, which corresponds to the case 
  $c_B'=0$ in assumption {\bf B}. Still, this case
 can be reconciled with our current setting by equivalently rewriting the
 doubly nonlinear equation in \eqref{eq_reg} as
  \[
  A_\eps(\partial_t u_\eps) + \tilde B(u_\eps) \ni c_B'u_\eps\,,
  \]
  where now $\tilde B:=B+c_B'I_H:V\to V^*$ is maximal $p$--strongly monotone on $V$.
\end{proof}

\subsection{Uniform estimates}\label{sec:ue}
First of all, we recall the energy balance \eqref{en_eps}
that reads
\begin{align*}
  &\Phi( u_\eps(t)) 
  + \int_0^t\left( v_\eps(s), \partial_t u^d_\eps(s)\right)_H\,\d s \\
  &=\Phi(u_0) 
  +\frac12\int_0^t\operatorname{Tr}\left[L(s, u_\eps(s))\right]\,\d s
  +\int_0^t\left(B( u_\eps(s)), G(s, u_\eps(s))\,\d  W(s)\right)_H\,.
\end{align*}
Taking into account the definition of $A_\eps$ and the monotonicity of $A$, we have that 
\[
  \left( v_\eps, \partial_t u^d_\eps\right)_H \geq 
  \eps\norm{\partial_t  u_\eps^d}_H^2
  +c_A\norm{\partial_t  u_\eps^d}_H\,.
\]
Moreover, assumption {\bf B} yields (see \cite[\S~4.2]{ScarStef-SDNL2}) that 
\[
  \Phi ( u_\eps) \geq \frac{c_B}{p}\norm{ u_\eps}_V^p - \frac{c_B'}2\norm{u_\eps}_H^2\,,
\]
so that, recalling that $c_B'=0$ if $p=2$, the Young inequality ensures that 
\[
  \Phi( u_\eps) \geq \frac{c_B}{2p}\norm{ u_\eps}_V^p - C\,.
\] 
Furthermore, by the assumption {\bf G} we get that 
\[
\operatorname{Tr}\left[L(\cdot, u_\eps)\right] \leq h_G(\cdot) + C_G\norm{ u_\eps}_V^p\,.
\]
Taking all these remarks into account, we infer that 
\begin{align*}
  &\frac{c_B}{2p}\norm{ u_\eps(t)}_V^p
  +\eps\int_0^t\norm{\partial_t u^d_\eps(s)}_H^2\,\d s
  +c_A\int_0^t\norm{\partial_t u^d_\eps(s)}_H\,\d s\\
  &\leq C  +\Phi (u_0) + \frac12\norm{h_G}_{L^1(0,T)} + 
  \frac{C_G}{2}\int_0^t\norm{ u_\eps(s)}_V^p\,\d s
 +\int_0^t
  \left(B( u_\eps(s)), G(s, u_\eps(s))\,\d W(s)\right)_H\,.
\end{align*}
Let now $\ell\geq2$ be arbitrary but fixed.
Taking supremum in $(0,t)$, $\ell$--power, and expectations yield,
for a positive constant $C$ independent of $\eps$,
\begin{align*}
  &\E\norm{u_\eps}_{L^\infty(0,t;V)}^{p\ell}
  +\eps^{\ell}\E\norm{\partial_t u^d_\eps}_{L^2(0,t; H)}^{2\ell}
  +\E\norm{\partial_t u^d_\eps}_{L^1(0,t; H)}^\ell\\
  &\leq C\left(1 + \int_0^t\E\norm{u_\eps}_{L^\infty(0,s;V)}^{p\ell}\,\d s
  +\E\sup_{r\in[0,t]}\left|\int_0^r
  \left(B(u_\eps(s)), G(s, u_\eps(s))\,\d W(s)\right)_H\right|^\ell\right)\,.
\end{align*}
Now,
by \eqref{eq2_var_eps} and assumption {\bf A} we have, 
considering $\eps\in(0,1)$ with no restriction, 
\[
  \norm{B(u_\eps)}_H=\norm{A_\eps(\partial_t u_\eps^d)}_H\leq
  \eps\norm{\partial_t u^d_\eps}_H + K_A\leq
  \eps^{1/2}\norm{\partial_t u^d_\eps}_H + K_A\,,
\]
so that by comparison in the last estimate we obtain
\begin{align*}
  &\E\norm{u_\eps}_{L^\infty(0,t;V)}^{p\ell}
  +\eps^{\ell}\E\norm{\partial_t u^d_\eps}_{L^2(0,t; H)}^{2\ell}
  +\E\norm{B(u_\eps)}^{2\ell}_{L^2(0,t; H)}
  +\E\norm{\partial_t u^d_\eps}_{L^1(0,t; H)}^\ell\\
  &\leq  C \left(1 + \int_0^t\E\norm{u_\eps}_{L^\infty(0,s;V)}^{p\ell}\,\d s
  +\E\sup_{r\in[0,t]}\left|\int_0^r
  \left(B(u_\eps(s)), G(s, u_\eps(s))\,\d W(s)\right)_H\right|^\ell\right)\,.
\end{align*}
As far as the stochastic integral on the right--hand side is concerned, 
we use the Burkholder--Davis--Gundy inequality combined with 
assumption {\bf G} and the 
weighted Young inequality to get, for every $\sigma>0$,
\begin{align*}
  &\E\sup_{r\in[0,t]}\left|\int_0^r
  \left(B(u_\eps(s)), G(s, u_\eps(s))\,\d W(s)\right)_H\right|^\ell\\
  &\leq  C  \E\left(\int_0^t\norm{B(u_\eps(s))}_H^2
  \norm{G(s,u_\eps(s))}_{\cL^2(U,H)}^2\,\d s\right)^{\ell/2}\\
  &\leq  C  \E\left[\norm{B(u_\eps)}_{L^2(0,t; H)}^\ell
  \left(1 + \norm{u_\eps}_{L^\infty(0,t; V)}^{\ell\nu}\right)\right]\\
  &\leq \sigma\E\norm{B(u_\eps)}^{2\ell}_{L^2(0,t; H)}
  + C_\sigma  \left(1 + \E\norm{u_\eps}_{L^\infty(0,t;V)}^{2\ell\nu}\right)\,.
\end{align*} 
Since $2\nu<p$ by assumption {\bf G}, using again the Young inequality yields 
\begin{align*}
  &\E\sup_{r\in[0,t]}\left|\int_0^r
  \left(B(u_\eps(s)), G(s, u_\eps(s))\,\d W(s)\right)_H\right|^\ell\\
  &\leq\sigma\E\norm{B(u_\eps)}^{2\ell}_{L^2(0,t; H)}
  +\sigma\E\norm{u_\eps}_{L^\infty(0,t;V)}^{p\ell} +  C_\sigma\,,
\end{align*}
so that 
choosing $\sigma$ sufficiently small 
and rearranging the terms we obtain 
\begin{align*}
  &\E\norm{u_\eps}_{L^\infty(0,t;V)}^{p\ell}
  +\eps^{\ell}\E\norm{\partial_t u^d_\eps}_{L^2(0,t; H)}^{2\ell}
  +\E\norm{B(u_\eps)}^{2\ell}_{L^2(0,t; H)}
  +\E\norm{\partial_t u^d_\eps}_{L^1(0,t; H)}^\ell\\
  &\leq  C \left(1 + \int_0^t\E\norm{u_\eps}_{L^\infty(0,s;V)}^{p\ell}\,\d s\right)\,.
\end{align*}
The Gronwall Lemma ensures then that there exists a constant $C_\ell>0$,
independent of $\eps$, such that 
\begin{align}
  \label{est1}
  \norm{u_\eps}_{L^{p\ell}_\cP(\Omega; L^\infty(0,T; V))}
  +\norm{\partial_t  u_\eps^d}_{L^\ell_{ \cP}( \Omega; L^1(0,T; H))}&\leq C_\ell\,,\\
  \label{est2}
  \eps^{1/2}\norm{\partial_t  u_\eps^d}_{L^{2\ell}_{ \cP}( \Omega; 
  L^2(0,T; H))} &\leq C_\ell\,,\\
  \label{est3}
  \norm{B(u_\eps)}_{L^{2\ell}_\cP(\Omega; L^2(0,T; H))} &\leq C_\ell\,.
\end{align}
From \eqref{est2}--\eqref{est3}, the definition of $A_\eps$
and the fact that $\eps\in(0,1)$
it follows also that 
\beq
  \label{est4}
  \norm{ v_\eps}_{L^{2\ell}_{ \cP}( \Omega; L^2(0,T; H))}\leq C_\ell\,.
\eeq
\section{Vanishing-viscosity limit:  Proof of Theorem~\ref{thm1}}
\label{sec:ri}

Moving from the viscous approximation of Section \ref{sec:limit},
we now prove Theorem  \ref{thm1} by passing to the limit as $\eps \to 0$. To this aim, we 
introduce time rescalings in 
  Subsection~\ref{sec:tr} and check the convergence of properly  parametrized 
  viscous 
solutions in Subsection~\ref{ssec:conv}. Eventually, we
\UUU prove \EEE that the limit of  parametrized  viscous solutions is a  parametrized 
martingale solution in Section \ref{sec:pa}.

\subsection{Time rescaling}\label{sec:tr}
Estimates \eqref{est1}--\eqref{est4} fall short of ensuring compactness on the sequence $(\partial_t u_\eps^d)_\eps$.
To this end, we now introduce a time rescaling on the solutions
and exploit the rate--independency of the evolution.

Define the rescaled final times as the random variables
\[
  \hat T_\eps:\Omega\to \erre_+\,,
  \qquad
  \hat T_\eps:=T + \int_0^T
  \norm{\partial_t  u_\eps(s)}_H \,\d s\,.
\]
Furthermore, define the arc length 
\[
  \tau_\eps:\Omega\times[0,T]\to\erre_+\,, \qquad
  \tau_\eps(t):=t + \int_0^t 
  \norm{\partial_t  u_\eps(s)}_H 
  \,\d s\,, \quad t\in[0,T]\,,
\]
and note that 
\[
  \tau_\eps \in L^\ell_\cP(\Omega; H^{1}(0,T)) \quad\forall\,\ell\geq1\,,
  \qquad 0\leq\tau_\eps(t)\leq \hat T_\eps \quad\forall\,t\in[0,T]\,,\quad\P\text{-a.s.}
\]

Now, redefine $\Omega_0\in\cF$ to be such that $\P(\Omega_0)=1$
and $\tau_\eps(\omega)\in C^0([0,T])$ for every $\omega\in\Omega_0$
and for every $\eps$ (one can take indeed a sequence $\eps\searrow0$ with no restriction).
For $\omega\in\Omega_0$, it is clear that
\[
\tau_\eps(\omega):[0,T] \to [0,\hat T_\eps(\omega)]
\]
is continuous, increasing, and surjective, hence also invertible
with well--defined inverse 
\[
\tau_\eps(\omega)^{-1}:[0,\hat T_\eps(\omega)]\to [0,T]\,.
\]
Hence, it makes sense to define the process
\[
  \hat t_\eps:\Omega\times\erre_+\to [0,T]\,,
  \qquad
  \hat t_\eps(\omega, \tau):=
  \begin{cases}
  \tau_\eps(\omega)^{-1}(\tau) \quad&\text{if } \tau\in[0,\hat T_\eps(\omega)]\,,
  \;\omega\in\Omega_0\,,\\
  T=\tau_\eps(\omega)^{-1}(\hat T_\eps(\omega))
  \quad&\text{if }\tau>\hat T_\eps(\omega)\,,\;\omega\in\Omega_0\,,\\
  0 \quad&\text{if } \omega\in\Omega\setminus\Omega_0\,.
  \end{cases}
\]
We collect the main properties of $\hat t_\eps$ in the following
Lemma, whose proof is omitted as it is almost identical to the
one of Lemma~\ref{lem:t_0}. 
\begin{lem}
  \label{lem:t}
  For every $\eps\in(0,1)$, $\hat t_\eps$ is a well--defined  
  progressively measurable continuous non-decreasing process. Moreover,
  $\hat t_\eps(\cdot,\tau):\Omega\to[0,T]$ is a bounded stopping time 
   for every $\tau\geq0$.
\end{lem}
% \begin{proof}
% The proof is identical to the one of Lemma~\ref{lem:t_0}.
% \end{proof}

At this point, we can define the rescaled processes
\begin{align*}
  \hat u_\eps: \Omega\times\erre_+\to V\,,& \qquad
  \hat u_\eps(\omega,\tau):=
  u_\eps(\omega,\hat t_\eps(\omega,\tau))\,, &&(\omega,\tau)\in\Omega\times\erre_+\,,\\
  \hat u_\eps^d: \Omega\times\erre_+\to H\,,& \qquad
  \hat u_\eps^d(\omega,\tau):=
  u_\eps^d(\omega,\hat t_\eps(\omega,\tau))\,, &&(\omega,\tau)\in\Omega\times\erre_+\,,\\
  \hat G_\eps:\Omega\times\erre_+\times H\to\cL^2(U,H)\,,&
  \qquad
  \hat G_\eps(\omega,\tau,z):=G(\hat t(\omega,\tau), z)\,, 
  &&(\omega,\tau,z)\in\Omega\times\erre_+\times H\,,\\
  \hat M_\eps: \Omega\times\erre_+\to U_1\,,& \qquad
  \hat M_\eps(\omega,\tau):=
  W(\omega,\hat t_\eps(\omega,\tau))\,, &&(\omega,\tau)\in\Omega\times\erre_+\,,\\
  \hat v_\eps:\Omega\times\erre_+\to H\,,&
  \qquad
  \hat v_\eps(\omega,\tau):=v_\eps(\omega,\hat t(\omega,\tau))\,, 
  &&(\omega,\tau)\in\Omega\times\erre_+\,.
\end{align*}
Since we have performed a time-rescaling of the processes, 
in order to preserve  the measurability
the filtration $(\cF_t)_{t\in[0,T]}$ needs to be rescaled as well.
To the end, recalling that $\hat t_\eps(\cdot,\tau)$ is a stopping time for every 
$\tau\geq0$ by Lemma~\ref{lem:t}, the $\sigma$-algebra $\cF_{\hat t_\eps(\cdot,\tau)}$
can be rigorously defined, and 
it makes sense to define the rescaled filtration 
\[
  \hat\cF_{\eps,\tau}:=\cF_{\hat t_\eps(\cdot, \tau)}=
  \left\{F\in\cF:\; F\cap\{\omega\in\Omega:\hat t_\eps(\omega,\tau)\leq t\}
  \in\cF_t\quad\forall\,t\in[0,T]\right\}\,, \qquad
  \hat\cF:=\cF\,.
\]
By reproducing in the viscous setting $\eps>0$ the arguments of Lemma \ref{lem:hat_0} we
obtain the following result, \UUU which corresponds to a stochastic
version of the original argument by Efendiev \& Mielke \cite{EM} and \EEE whose proof is omitted. 
\begin{lem}
 \label{lem:hat}
 For every $\eps\in(0,1)$, the following holds:
 \begin{enumerate}[(i)]
 \item\label{lem:hat1} $(\hat\cF_{\eps,\tau})_{\tau\geq0}$
 is a filtration satisfying the usual conditions, with associated 
 progressive $\sigma$-algebra denoted by $\hat \cP_\eps$.
 \item\label{lem:hat2} 
 $\hat t_\eps$, $\hat u_\eps$, $\hat u_\eps^d$, and $\hat v_\eps$ are progressively measurable
 on $(\Omega,\hat \cF,(\hat\cF_{\eps,\tau})_{\tau\geq0},\P)$, and satisfy 
 \begin{align*}
 \hat t_\eps\in W^{1,\infty}_{loc}(0,+\infty) \qquad&\P\text{-a.s.}\,,\\
 \hat u_\eps\in C^0([0,+\infty); H)\cap L^\infty(0,+\infty; V) \qquad&\P\text{-a.s.}\,,\\
 \hat u_\eps^d\in W^{1,\infty}_{loc}(0,+\infty; H)\,, \qquad&\P\text{-a.s.}\,,\\
 \partial_t \hat u_\eps^d = \partial_t u^d_{\eps}(\hat t_\eps)\hat t_\eps'\,,
 \qquad&\P\text{-a.s.}\,,\\
 \hat t_\eps'(\tau) + \norm{\partial_t\hat u_\eps^d(\tau)}_H=1
 \qquad&\text{for a.e.~}\tau>0\,,\quad\P\text{-a.s.}\,,\\
 \hat v_\eps\in L^1_{loc}(0,+\infty; H)\,, \qquad&\P\text{-a.s.}\,,
 \end{align*}
 \item\label{lem:hat3} $\hat M_\eps$ is a continuous square-integrable
 $U$-cylindrical martingale 
 on $(\Omega,\hat \cF,(\hat\cF_{\eps,\tau})_{\tau\geq0},\P)$, with 
 tensor quadratic variation given by 
 \[
 \qqv{\hat M_\eps}(\tau)=Q_1\hat t_\eps(\tau)\,, \qquad\tau\geq0\,.
 \]
 \item\label{lem:hat4} For every $Y\in L^2_\cP(\Omega; L^2(0,T; \cL^2(U,H)))$,
 the rescaled process $\hat Y_\eps:=Y(\hat t_\eps)$ is 
 stochastically integrable with respect to $\hat M_\eps$ 
 on $(\Omega,\hat \cF,(\hat\cF_{\eps,\tau})_{\tau\geq0},\P)$ and
 \[
 \int_0^{\hat t_\eps(\tau)}Y(s)\,\d W(s)=
 \int_0^\tau \hat Y_\eps(\sigma)\,\d\hat M_\eps(\sigma) \qquad
 \forall\,\tau\geq0\,,\quad\P\text{-a.s.}
 \]
 In particular,
 $\hat G_\eps(\cdot, \hat u_\eps)$ is 
 stochastically integrable with respect to $\hat M_\eps$ 
 and it holds that
 \[
 \int_0^{\hat t_\eps(\tau)}G(s,u_\eps(s))\,\d W(s)=
 \int_0^\tau \hat G_\eps(\sigma, \hat u_\eps(\sigma))\,\d\hat M_\eps(\sigma) \qquad
 \forall\,\tau\geq0\,,\quad\P\text{-a.s.}
 \]
 \end{enumerate} 
\end{lem}

From the proof of Lemma~\ref{lem:hat_0} we also obtain that
 \beq\label{eq_sup}
  \sup_{\tau\geq0}\norm{\hat u_\eps(\omega,\tau)}_V=
  \sup_{t\in[0,T]}\norm{u_\eps(t)}_V \qquad\P\text{-a.s.}\,,
  \eeq
and
 \beq
  \label{est0_hat}
  \hat t_\eps'(\tau) + \norm{\partial_t u_\eps^d(\hat t_\eps(\tau))}_H\hat t_\eps'(\tau)=1
  \qquad
  \text{for a.e.~}\tau>0\,,\quad\P\text{-a.s.}
\eeq
Consequently, Lemma~\ref{lem:hat}, 
the identity \eqref{eq_sup}, and the regularity of $u_\eps$
yield 
\[
  \hat u_\eps \in L^p_{\cP}(\Omega; C^0([0,+\infty); H))\cap 
  L^p_w(\Omega;L^\infty(0,+\infty; V))\,,
\]
where
by the estimate \eqref{est1}
we have
\beq
  \label{est1_hat}
  \norm{\hat u_\eps}_{L^{pl}_{\cP}(\Omega; L^\infty(0,+\infty; V))} \leq C_\ell\,,
\eeq
Moreover, 
Lemma~\ref{lem:hat} also ensures that
\beq
  \label{est2_hat}
  \hat t'_\eps(\tau) + \norm{\partial_t \hat u_\eps(\tau)}_H = 1
  \qquad
  \text{for a.e.~}\tau>0\,,\quad\P\text{-a.s.}
\eeq
It follows in particular that the pair 
$(\hat t_\eps, \hat u_\eps):\Omega\times[0,+\infty)\to \erre\times H$
is $1$--Lipschitz--continuous locally in time, uniformly in $\Omega$ and $\eps$: 
more specifically, 
for every $\hat T>0$ there exists $C_{\hat T}>0$, independent of $\eps$, such that
\beq
  \label{est4_hat}
  \norm{(\hat t_\eps, \hat u_\eps^d)}_{
  L^\infty_\cP(\Omega; W^{1,\infty}(0,\hat T; \erre\times H))} \leq C_{\hat T}\,.
\eeq
In addition, by assumption {\bf G} it holds, for all $\sigma\in[0,\tau]$,
\[
  \norm{\hat G_{\eps}(\sigma, \hat u_{\eps}(\sigma))}_{\cL^2(U,H)}
  =\norm{G(\hat t_{\eps}(\sigma), \hat u_{\eps}(\sigma))}_{\cL^2(U,H)}
  \leq C_G\left(1+\norm{\hat u_{\eps}(\sigma)}_V^{\nu}\right)\,,
\]
from which, by estimate \eqref{est1_hat},
\beq
  \label{est3_hat}
  \norm{\hat G_\eps(\cdot, \hat u_\eps)}_{
  L^{p\ell/\nu}_\cP(\Omega; C^0([0,T]; \cL^2(U,H)))}
  \leq C_\ell\,.
\eeq

Let us now reformulate  problem \eqref{eq1_var_eps}--\eqref{incl_eps}
in rescaled time. 
First of all, it follows from \eqref{eq1_var_eps} and the time rescaling in Lemma~\ref{lem:hat}
that 
\beq
  \label{eq1_hat_eps}
  \hat u_\eps(\tau)= u_0 + \int_0^\tau\partial_t \hat u_\eps^d(\sigma)\,\d\sigma
  +\int_0^\tau\hat G_\eps(\sigma, \hat u_\eps(\sigma))\,\d\hat M_\eps(\sigma)
  \qquad\forall\,\tau\geq0\,,\quad\P\text{-a.s.}
\eeq
Secondly, from \eqref{eq2_var_eps} 
we have that
\beq
  \label{eq2_hat_eps}
  \hat v_\eps(\tau) + B(\hat u_\eps(\tau)) = 0 \qquad\text{for a.e.~}\tau>0\,,\quad\P\text{-a.s.}
\eeq
Now, \UUU the differential inclusion \eqref{incl_eps} reads \EEE
\[
  \eps\partial_t u_\eps^d + A(\partial_t u_\eps^d) + B(u_\eps) \ni 0\,,
\]
and from \eqref{est2_hat} we know that 
\[
  \partial_t u_\eps^d(t) = 
  \frac1{1 - \norm{\partial_t \hat u_\eps^d(\tau_\eps(t))}}_H
  \partial_t \hat u_\eps^d(\tau_\eps(t))
  \qquad\text{for a.e.~}t\in(0,T)\,.
\]
Together with the $0$--homogeneity of $A$ this implies that
\begin{equation}\label{questa}
  \frac1{1 - \norm{\partial_t \hat u_\eps^d(\tau)}}_H
  \eps\partial_t \hat u_\eps^d(\tau)+
  A(\partial_t \hat u_\eps^d(\tau))
  +B(\hat u_\eps(\tau))\ni 0 
  \qquad\text{for a.e.}~\tau>0\,.
\end{equation}
In order to write this in a more compact form, 
it is natural to introduce the proper convex lower semicontinuous function 
\[
 \mathcal F:\erre\to[0,+\infty]\,,\qquad
  \mathcal F(r):=\begin{cases}
  - r - \ln(1- r) \quad&\text{if } r\in[0,1)\,,\\
  +\infty \quad&\text{otherwise}\,,
  \end{cases}
\]
whose subdifferential is given by 
\[
  f:=\partial \mathcal F:\erre\to2^\erre\,, \qquad
  f(r)=\begin{cases}
  (-\infty,0] \quad&\text{if } r=0\,,\\
  \frac1{1-r} - 1 \quad&\text{if } r\in(0,1)\,,\\
  \emptyset \quad&\text{otherwise}\,,
  \end{cases}
\]
and set 
\[
  \widehat \Psi_\eps:H\to[0,+\infty]\,,
  \qquad
  \widehat\Psi_\eps(z):= \Psi(z) + \eps \mathcal F(\norm{z}_H)\,, \quad z\in H\,.
\]
It is indeed immediate to check that the subdifferential 
of $\widehat\Psi_\eps$ is given by
\[
  \partial\widehat\Psi_\eps:H\to2^H\,,
  \qquad\partial\widehat\Psi_\eps(z) =
  \begin{cases}
  A(z) + \eps f(\norm{z}_H)\frac{z}{\norm{z}_H} \quad&\text{if } z\neq0\,,\\
  A(z) \quad&\text{if } z=0\,.
  \end{cases}
\]
Taking this into account, we may rewrite the nonlinear inclusion  
\eqref{eq2_var_eps} as 
\beq
  \label{incl_hat_eps}
  \hat v_\eps(\tau)\in\partial\widehat\Psi_\eps(\partial_t \hat u_\eps^d(\tau))
  \qquad\text{for a.e.~}\tau>0\,,\quad\P\text{-a.s.}
\eeq

Testing equation \eqref{questa} by $\partial_t \hat u_\eps^d$ and using
assumption {\bf A} and \eqref{est2_hat}, we get 
\begin{equation}\label{questa2}
  \eps f(\norm{\partial_t \hat u_\eps^d}_H)\norm{\partial_t \hat u_\eps^d}_H
  +c_A\norm{\partial_t \hat u_\eps^d}_H = 
  -\left(B(\hat u_\eps), \partial_t \hat u_\eps^d\right)
  \leq\norm{B(\hat u_\eps)}_H\,.
\end{equation}
For every $\hat T>0$, using the definition of $\tau_\eps$ and equation \eqref{eq2_var_eps},
\begin{align*}
  \int_0^{\hat T}\norm{B(\hat u_\eps(\tau))}_H\,\d\tau&=
  \int_0^{\hat t_\eps(\hat T)}\norm{B(u_\eps(t))}_H\tau_\eps'(t)\,\d t
  \leq\int_0^T\norm{B(u_\eps(t))}_H\norm{\partial_t u_\eps^d(t)}_H\,\d t\\
  &=\int_0^T\norm{v_\eps(t)}_H\norm{\partial_t u_\eps^d(t)}_H\,\d t\,.
\end{align*}
Since $v_\eps-\eps\partial_t u_\eps^d\in A(\partial_t u_\eps^d)$ by equation \eqref{incl_eps},
using assumption {\bf A} we get
\[
  \norm{v_\eps}_H\leq \eps\norm{\partial_tu_\eps^d}_H + K_A\,.
\]
Hence, by integrating inequality \eqref{questa2} over $[0,\hat
T]$ we get
\begin{align*}
  \eps \int_0^{\hat T}
  f(\norm{\partial_t \hat u_\eps^d(\tau)}_H)\norm{\partial_t \hat u_\eps^d(\tau)}_H\,\d\tau&\leq
  \int_0^{\hat T}\norm{B(\hat u_\eps(\tau))}_H\,\d\tau\\
  &\leq\int_0^T\left(\eps\norm{\partial_t u_\eps^d(t)}_H^2 + \norm{\partial_t u_\eps^d(t)}_H\right)\,\d t
  \qquad\forall\,\hat T>0\,.
\end{align*}
Since $\hat T>0$ is arbitrary, estimates \eqref{est1}--\eqref{est2} imply that 
\beq
  \label{est5_hat}
  \eps\norm{f(\norm{\partial_t \hat u_\eps^d}_H)
  \partial_t \hat u_\eps^d}_{L^\ell_\cP(\Omega; L^1(0,+\infty; H))}
  +\norm{B(\hat u_\eps)}_{L^\ell_\cP(\Omega; L^1(0,+\infty; H))}\leq C_\ell\,,
\eeq
which ensure in particular that 
\beq
  \label{est6_hat}
  \eps\norm{f(\norm{\partial_t \hat u_\eps^d}_H)}_{L^\ell_\cP(\Omega; L^1(0,+\infty))}
  +\norm{\hat v_\eps}_{L^\ell_\cP(\Omega; L^1(0,+\infty;H))}
  \leq C_\ell\,,
\eeq
while \eqref{est1_hat} and assumption {\bf B} give 
\beq
  \label{est7_hat}
  \norm{B(\hat u_\eps)}_{L^{q\ell}_\cP(\Omega; L^\infty(0,+\infty; V^*))}\leq C_{\ell}\,.
\eeq

\subsection{Convergences of the rescaled processes}
\label{ssec:conv}
Taking estimate \eqref{est1_hat} into account, 
we deduce that there exists
\begin{align*}
  \hat u \in \bigcap_{\hat T>0}
  L^{p\ell}_w(\Omega; L^\infty(0,\hat T; V))
\end{align*}
such that 
\beq
  \label{conv1_hat}
  \hat u_\eps \wstarto \hat u \qquad\text{in } 
   L^{p\ell}_w(\Omega; L^\infty(0,\hat T; V)) \qquad\forall\,\hat T>0\,.
\eeq
Moreover, noting that $\hat\cP_\eps\subseteq\cP$, 
from \eqref{est4_hat} and the Ascoli--Arzel\`a 
and Aubin--Lions \cite{Simon} Theorems,
we infer that there exist  $\cP$-measurable processes
\[
  (\hat t, \hat u^d):\Omega\times\erre_+\to [0,T]\times H
\]
such that 
\[
  (\hat t, \hat u^d) \in W^{1,\infty}_{loc}(0,+\infty; \erre\times H) \qquad\P\text{-a.s.}
\]
and, for every $\hat T>0$,
\begin{align}
  \label{conv2_hat}
  \hat t_\eps \to \hat t \qquad&\text{in } C^0([0,\hat T]) \quad\P\text{-a.s.}\,,\\
  \label{conv2'_hat}
  \hat t_\eps \to \hat t \qquad&\text{in } L^r_\cP(\Omega;
  C^0([0,\hat T])) \quad\forall\,r\geq1\,,\\
  \label{conv3_hat}
  \hat t_\eps'\wstarto \hat t' \qquad&\text{in } L^\infty_\cP(\Omega\times(0,\hat T))\,,\\
  \label{conv4_hat}
  \hat u_\eps^d \to \hat u^d \qquad&\text{in } C^0([0,\hat T]; E^*) \quad\P\text{-a.s.}\,,\\
  \label{conv4'_hat}
  \hat u_\eps^d \to \hat u^d \qquad&\text{in }
                                     L^r_\cP(\Omega;C^0([0,\hat T];
                                      E^* ))
  \quad\forall\,r\geq1\,,\\
  \label{conv5_hat}
  \hat u_\eps^d \wstarto \hat u^d \qquad&\text{in } 
  L^\infty_\cP(\Omega\times(0,\hat T); H)\,,\\
  \label{conv6_hat}
  \partial_t\hat u_\eps^d \wstarto \partial_t\hat u^d 
  \qquad&\text{in } L^\infty_\cP(\Omega\times(0,\hat T); H)\,.
\end{align}
Furthermore, since $W$ has $\alpha$-H\"older-continuous trajectories 
in $U_1$ for every $\alpha\in(0,1/2)$,
we have for every $\eps_1,\eps_2>0$ and $\tau\in[0,\hat T]$ that
\begin{align*}
  \|(\hat M_{\eps_1}-\hat M_{\eps_2})(\tau)\|_{U_1}
  &=\|W(\hat t_{\eps_1}(\tau)) - W(\hat t_{\eps_2}(\tau))\|_{U_1}
  \leq\|W\|_{C^{0,\alpha}([0,\hat T]; U_1)}|\hat t_{\eps_1}(\tau)-\hat t_{\eps_2}(\tau)|^\alpha
  \\
  &\leq\|W\|_{C^{0,\alpha}([0,\hat T]; U_1)}
  \norm{\hat t_{\eps_1}-\hat t_{\eps_2}}_{C^0([0,\hat T])}^\alpha\,.
\end{align*}
Thanks to the convergences \eqref{conv2_hat}--\eqref{conv3_hat}
and the finiteness  of every moment of $W$ these imply
that  there exists 
a $\cP$-progressively measurable process
\[
  \hat M:\Omega\times\erre_+\to U_1\,,
  \qquad
  \hat M\in \bigcap_{\hat T>0}\bigcap_{r\geq1}
  L^r_\cP(\Omega; C^0([0,\hat T]; U_1))
\]
such that, for every $\hat T>0$,
\beq\label{conv7_hat}
  \hat M_\eps\to \hat M \qquad\text{in } L^r_\cP(\Omega; C^0([0,\hat T]; U_1))
  \quad\forall\,r\geq1\,.
\eeq 
It is  hence  natural to define the limiting rescaled process
\[
  \hat G:\Omega\times\erre_+\times H\to \cL^2(U,H)\,,
  \qquad
  \hat G(\omega,\tau, z):=G(\omega,\hat t(\omega,\tau), z)\,,
  \quad(\omega,\tau,z)\in\Omega\times\erre_+\times H\,,
\]
and analogously 
\begin{alignat*}{2}
  \hat L_\eps:\Omega\times\erre_+\times V\to\cL^1(H,H)\,,
  \qquad&\hat L_\eps(\omega,\tau,z):=L(\hat t_\eps(\omega,t),z)\,,
  \quad&&(\omega,\tau,z)\in\Omega\times\erre_+\times V\,,\\
  \hat L:\Omega\times\erre_+\times V\to\cL^1(H,H)\,,
  \qquad &\hat L(\omega,\tau,z):=L(\hat t(\omega,t),z)\,,
  \quad&&(\omega,\tau,z)\in\Omega\times\erre_+\times V\,.
\end{alignat*}

We prove now some strong compactness properties on
the rescaled processes.
\begin{lem}
  \label{lem:strong_conv}
  The following convergences hold for every $\hat T>0$:
  \begin{alignat}{2}
  \label{conv8_hat}
  \hat u_\eps\to \hat u
  \qquad&\text{in } L^{r}_\cP(\Omega; C^0([0,\hat T]; H))
  \quad&&\forall\,r\in[1,p\ell)\,,\\
  \label{conv9_hat}
  \hat G_\eps(\cdot, \hat u_\eps)\to \hat G(\cdot, \hat u)
  \qquad&\text{in } L^{r}_\cP(\Omega; C^0([0,\hat T]; \cL^2(U,H)))
  \quad&&\forall\,r\in[1,p\ell)\,,\\
  \label{conv11_hat}
  \hat u_\eps \to \hat u \qquad&\text{in } L^r_\cP(\Omega; L^p(0,\hat T; V))
  \quad&&\forall\,r\in\left[1,\textstyle\frac{p^2\ell}{p+1}\right)\,,\\
  \label{conv15_hat}
  \hat L_\eps(\cdot, \hat u_\eps)\to \hat L(\cdot, \hat u)
  \qquad&\text{in } L^1_{\cP}(\Omega; L^1(0,\hat T; \cL^1(H,H)))\,,\\
  \label{conv12_hat}
  B(\hat u_\eps)\to B(\hat u)
  \qquad&\text{in } L^r_\cP(\Omega; L^p(0,\hat T; V^*))
  \quad&&\forall\,r\in\textstyle\left[1,\frac{p^2\ell}{p^2-p+1}\right)\,,\\
  \label{conv13_hat}
  B(\hat u_\eps)\wstarto B(\hat u)
  \qquad&\text{in } L^{q\ell}_w(\Omega; L^\infty(0,\hat T; V^*))
  \end{alignat}
In particular,  owing to \cite[Thm.~2.1]{Strauss} we have that $\hat u \in
C^0_w([0,T],V)$, $\P$-a.s. and
\begin{equation}
  \label{conv13_hat_due}
  \hat u_\eps(\tau) \wto \hat u(\tau) \qquad \text{in } L^p(\Omega,\cF_\tau;V), \quad \forall
  \tau \in [0,\hat T].
\end{equation}
Moreover,
setting   $\hat v:=-B(\hat u)$, it holds that
\begin{alignat}{2}
  \label{conv14_hat}
  \hat v_\eps \to \hat v \qquad&\text{in }
  L^r_\cP(\Omega; L^p(0,\hat T; V^*))
  \quad&&\forall\,r\in\textstyle\left[1,\frac{p^2\ell}{p^2-p+1}\right)\,,\\
  \label{conv14_hat'}
  \hat v_\eps \wstarto \hat v \qquad&\text{in }
  L^{q\ell}_w(\Omega; L^\infty(0,\hat T; V^*))\,.
\end{alignat}
\end{lem}
\begin{proof}
{\it Proof of \eqref{conv8_hat}--\eqref{conv9_hat}.}
Subtracting \eqref{eq1_hat_eps} at some arbitrary 
$\eps_1,\eps_2\in(0,1)$ we have 
\beq
  \label{conv8_aux}
  \hat u_{\eps_1}-\hat u_{\eps_2}=
  \hat u_{\eps_1}^d - \hat u_{\eps_2}^d + 
  \hat G_{\eps_1}(\cdot, \hat u_{\eps_1})\cdot\hat M_{\eps_1}
  -\hat G_{\eps_2}(\cdot, \hat u_{\eps_2})\cdot\hat M_{\eps_2}\,.
  \eeq
   In order to control the right-hand side we use the 
  Burkholder--Davis--Gundy inequality, which implies that
  for all $\hat T>0$, $\tau\in[0,\hat T]$, and any $2\leq
 r$,
\begin{align*}
  &\norm{\hat G_{\eps_1}(\cdot, \hat u_{\eps_1})\cdot\hat M_{\eps_1}
  -\hat G_{\eps_2}(\cdot, \hat u_{\eps_2})\cdot\hat M_{\eps_2}
  }_{L^r_\cP(\Omega; C^0([0,\tau]; H))}^r\\
  &\leq
   C \norm{( \hat G_{\eps_1}(\cdot, \hat u_{\eps_1})
  -\hat G_{\eps_2}(\cdot, \hat u_{\eps_2}))\cdot\hat M_{\eps_2}
  }_{L^r_\cP(\Omega; C^0([0,\tau]; H))}^\ell\\
  &\qquad+ C \E\norm{\hat G_{\eps_2}(\cdot, \hat u_{\eps_2})
  \cdot(\hat M_{\eps_1}-\hat M_{\eps_2})
  }_{L^r_\cP(\Omega; C^0([0,\tau]; H))}^\ell\\
  &\leq  C \E\left(\int_0^\tau\norm{\hat G_{\eps_1}(\sigma, \hat u_{\eps_1}(\sigma))
  -\hat G_{\eps_2}(\sigma, \hat u_{\eps_2}(\sigma))}_{\cL^2(U,H)}^2
  \hat t'_{\eps_2}(\sigma)\,\d\sigma\right)^{r/2}\\
  &\qquad+ C \E\left(\int_0^\tau\norm{
  \hat G_{\eps_2}(\sigma, \hat u_{\eps_2}(\sigma))}_{\cL^2(U,H)}^2
  \,\d\qv{\hat M_{\eps_1}-\hat M_{\eps_2}}(\sigma)\right)^{r/2}=:I_1+I_2\,.
\end{align*}
We can handle the term $I_1$ thanks to assumption {\bf G} 
and condition \eqref{est2_hat},  for 
we have, for all $\sigma\in[0,\tau]$,
\begin{align*}
  &\norm{\hat G_{\eps_1}(\sigma, \hat u_{\eps_1}(\sigma))-
  \hat G_{\eps_2}(\sigma, \hat u_{\eps_2}(\sigma))}_{\cL^2(U,H)}^2
  \hat t_{\eps_1}'(\sigma)\\
  &\qquad\leq
  \norm{G(\hat t_{\eps_1}(\sigma), \hat u_{\eps_1}(\sigma))-
  G(\hat t_{\eps_2}(\sigma), \hat u_{\eps_2}(\sigma))}_{\cL^2(U,H)}^2\\
  &\qquad\leq 2C_G^2\left(
  \norm{(\hat u_{\eps_1}-\hat u_{\eps_2})(\sigma)}_H^2
  +|(\hat t_{\eps_1}-\hat t_{\eps_2})(\sigma)|^{2\alpha_G}\right)\,,
\end{align*}
so that%, renominating $c_r$,
\begin{align*}
  &I_1=c_r
  \E\left(\int_0^\tau\norm{\hat G_{\eps_1}(\sigma, \hat u_{\eps_1}(\sigma))
  -\hat G_{\eps_2}(\sigma, \hat u_{\eps_2}(\sigma))}_{\cL^2(U,H)}^2
  \hat t'_{\eps_1}(\sigma)\,\d\sigma\right)^{r/2}\\
  &\leq  C \int_0^\tau\E
  \norm{\hat u_{\eps_1}-\hat u_{\eps_2}}_{C^0([0,\sigma]; H)}^r\,\d\sigma
  + C_{\hat T} \E\norm{\hat t_{\eps_1}-\hat t_{\eps_2}}_{C^0([0,\hat T])}^{r\alpha_G}\,.
\end{align*}
For the  term $I_2$,  using the H\"older inequality  twice  
we get 
\begin{align*}
I_2&=\E\left(\int_0^\tau\norm{
  \hat G_{\eps_2}(\sigma, \hat u_{\eps_2}(\sigma))}_{\cL^2(U,H)}^2
  \,\d\qv{\hat M_{\eps_1}-\hat M_{\eps_2}}(\sigma)\right)^{r/2}\\
  &\leq\E\left(
  \|\hat G_{\eps_2}(\cdot, \hat u_{\eps_2})\|_{C^0([0,\hat T]; \cL^2(U,H))}^{r}
  \qv{\hat M_{\eps_1}-\hat M_{\eps_2}}(\tau)^{r/2}
  \right)\\
  &\leq\norm{\hat G_{\eps_2}(\cdot, \hat u_{\eps_2})}_{
  L^{2r}(\Omega; C^0([0,\hat T];\cL^2(U,H)))}^{r}
  \norm{\qv{\hat M_{\eps_1}-\hat M_{\eps_2}}(\tau)^{1/2}}_{L^{2r}(\Omega)}^{r}\,.
\end{align*}
Now, choosing $\ell\geq2$, by assumption {\bf G} it is not difficult
to check that $p\ell/\nu\geq4$, so that there exists $r\geq 2$
such that $2r\leq p\ell/\nu$. For such $r$, 
the estimate \eqref{est3_hat} and the Burkholder--Davis--Gundy inequality yield then
\begin{align*}
  I_2&\leq 
   C \norm{\qv{\hat M_{\eps_1}-\hat M_{\eps_2}}(\tau)^{1/2}}_{L^{2r}(\Omega)}^{r}
  \leq  C \norm{\hat M_{\eps_1}-\hat M_{\eps_2}}_{
  L^{2r}(\Omega; C^0([0,\hat T]; U_1))}^{r}\,.
\end{align*}
Putting this information together  
we deduce from \eqref{conv8_aux}
that, for all $\ell>2$ and $r\in[2,p\ell/(2\nu)]$,  
\begin{align*}
  &\|\hat u_{\eps_1}-\hat u_{\eps_2}\|^r_{L^r_\cP(\Omega;
    C^0([0,\tau];  E^*))}\\
  &\leq  C  
  \|\hat u^d_{\eps_1}-\hat u^d_{\eps_2}\|^r_{L^r_\cP(\Omega; C^0([0,\hat
    T];  E^* ))}
  + C_{\hat T}  \norm{\hat t_{\eps_1}-\hat t_{\eps_2}}_{
  L_\cP^{r\alpha_G}(\Omega; C^0([0,\hat T]))}^{r\alpha_G}\\
  &\qquad+ C_{\hat T}  \norm{\hat M_{\eps_1}-\hat M_{\eps_2}}_{
  L^{2r}(\Omega; C^0([0,\hat T]; U_1))}^{r}
  + C_{\hat T}  \int_0^\tau
  \norm{\hat u_{\eps_1}-\hat u_{\eps_2}}_{
  L^r_\cP(\Omega; C^0([0,\sigma]; H))}^r\,\d\sigma\,.
\end{align*}
Now,  we use assumption \eqref{eq:appro}, letting 
$\eta>0$ there be arbitrary but fixed. Noting that $p\ell\geq4$ and
picking $r\in[2,p\ell/2]$ only, 
it follows from  estimate \eqref{est1_hat} that
\begin{align*}
  &\|\hat u_{\eps_1}-\hat u_{\eps_2}\|^r_{L^r_\cP(\Omega;
    C^0([0,\tau];  E^*))}\\
  &\leq  C  
  \|\hat u^d_{\eps_1}-\hat u^d_{\eps_2}\|^r_{L^r_\cP(\Omega; C^0([0,\hat
    T];  E^* ))}
  + C_{\hat T}  \norm{\hat t_{\eps_1}-\hat t_{\eps_2}}_{
  L_\cP^{r\alpha_G}(\Omega; C^0([0,\hat T]))}^{r\alpha_G}\\
  &\quad+ C_{\hat T}  \norm{\hat M_{\eps_1}-\hat M_{\eps_2}}_{
  L^{2r}(\Omega; C^0([0,\hat T]; U_1))}^{r}
  + C_{\hat T}  \eta^r + \frac{ C_{\hat T}
    }{\eta^{\alpha r}} \int_0^\tau
  \|\hat u_{\eps_1}-\hat u_{\eps_2}\|^r_{
  L^r_\cP(\Omega; C^0([0,\sigma];  E^* ))} \,\d\sigma\,.
\end{align*}
Choose now $\tau$ so small that $C_{\hat T}
\tau/\eta^{\alpha r}<1/2$. From the latter inequality we hence get
\begin{align*}
  &\frac12\|\hat u_{\eps_1}-\hat u_{\eps_2}\|^r_{L^r_\cP(\Omega;
    C^0([0,\tau];  E^*))}\\
  &\leq  C_{\hat T}  \bigg(\eta^r+
  \norm{\hat u^d_{\eps_1}-\hat u^d_{\eps_2}}_{L^r_\cP(\Omega; C^0([0,\hat T]; E^*))}^r\\
  &+\norm{\hat t_{\eps_1}-\hat t_{\eps_2}}_{
  L_\cP^{r\alpha_G}(\Omega; C^0([0,\hat T]))}^{r\alpha_G}
  +\norm{\hat M_{\eps_1}-\hat M_{\eps_2}}_{
  L^{2r}(\Omega; C^0([0,\hat T]; U_1))}^{r}
  \bigg)\,.
\end{align*}
We now repeat the argument on the intervals $[\tau,2\tau]$,
$[2\tau,3\tau]$ $\dots$ and sum up the corresponding inequalities.
Note that, as $\tau <\eta^{\alpha r}/(2 C_{\hat T})$, we have  $ C
\eta^{-\alpha r}$ of such inequalities. One hence
has
\begin{align*}
  &\frac12\|\hat u_{\eps_1}-\hat u_{\eps_2}\|^r_{L^r_\cP(\Omega;
    C^0([0,\hat T];  E^*))}\\
  &\leq  C_{\hat T}  \eta^{r-\alpha r} + C_{\hat T}
    \eta^{-\alpha r}\bigg(
  \norm{\hat u^d_{\eps_1}-\hat u^d_{\eps_2}}_{L^r_\cP(\Omega; C^0([0,\hat T]; E^*))}^r\\
  &+\norm{\hat t_{\eps_1}-\hat t_{\eps_2}}_{
  L_\cP^{r\alpha_G}(\Omega; C^0([0,\hat T]))}^{r\alpha_G}
  +\norm{\hat M_{\eps_1}-\hat M_{\eps_2}}_{
  L^{2r}(\Omega; C^0([0,\hat T]; U_1))}^{r}
  \bigg)\,.
\end{align*}
Taking now $\eps_1,\, \eps_2\to 0$ and using the convergences \eqref{conv2'_hat},
\eqref{conv4'_hat}, and \eqref{conv7_hat} gives
\begin{align*}
  &\limsup_{\eps_1,\, \eps_2\to 0}\frac12\|\hat u_{\eps_1}-\hat u_{\eps_2}\|^r_{L^r_\cP(\Omega;
    C^0([0,\hat T]; E^*))}\leq  C_{\hat T}  \eta^{
    r-\alpha r}\,
\end{align*}
which in particular ensures that $\hat u_{\eps} \to \hat u$ strongly $L^r_\cP(\Omega; C^0([0,\hat T];  E^*))$  by arbitrariness of $\eta$ and the fact that $\alpha\in(0,1)$. 
By using again assumption \eqref{eq:appro} we deduce that
$$
\|\hat u_{\eps}-\hat u \|_{L^r_\cP(\Omega; C^0([0,\hat T];
  H))}\leq \eta \|\hat u_{\eps}-\hat u\|_{L^r_w(\Omega;
  L^\infty(0,\hat T;  V))}+
  \frac{c_E}{\eta^\alpha }\|\hat u_{\eps}-\hat u\|_{L^r_\cP(\Omega; C^0([0,\hat T];  E^*))}\,.
$$
By taking $\eps \to 0$ we hence have that 
$$\limsup_{\eps\to 0}\|\hat u_{\eps}-\hat u \|_{L^r_\cP(\Omega; C^0([0,\hat T];
  H))}\leq \eta \limsup_{\eps \to 0}\|\hat u_{\eps}-\hat u\|_{L^r_w(\Omega;
  L^\infty(0,\hat T;  V))}.$$
Since   $\hat u_\eps$ is bounded in $L^r_w(\Omega;
L^\infty(0,\hat T;  V))$ and $\eta$ is arbitrary, we have proved that
$\hat u_{\eps} \to \hat u$ strongly $L^r_\cP(\Omega; C^0([0,\hat T];
H))$. In particular, one has that \EEE  
\[
  \hat u_\eps\to \hat u
  \qquad\text{in } L^{p\ell/2}_\cP(\Omega; C^0([0,\hat T]; H))\,.
\]
This in turn implies \eqref{conv8_hat}
by the boundedness \eqref{est1_hat} and the Severini-Egorov theorem.
Also, by assumption {\bf G}, again \eqref{conv2'_hat}, and \eqref{conv8_hat}
we readily get also \eqref{conv9_hat}.\\
{\it Proof of \eqref{conv11_hat}.}
Using assumption {\bf B}, 
equation \eqref{eq2_hat_eps}, and relation \eqref{est2_hat} we obtain 
\begin{align*}
  c_B\norm{\hat u_{\eps_1}-\hat u_{\eps_2}}_V^p-
  c_B'\norm{\hat u_{\eps_1}-\hat u_{\eps_2}}_H^2&\leq
  \left(B(\hat u_{\eps_1})- B(\hat u_{\eps_2}), \hat u_{\eps_1}-\hat u_{\eps_2}\right)_H\\
  &=\left(\hat v_{\eps_2}
  -\hat v_{\eps_1}, 
  \hat u_{\eps_1}-\hat u_{\eps_2}\right)_H\\
  &\leq \left(\norm{\hat v_{\eps_1}}_H+\norm{\hat v_{\eps_2}}_H\right)
  \norm{\hat u_{\eps_1}-\hat u_{\eps_2}}_H\,,
\end{align*}
from which, by the H\"older inequality, for every $\hat T>0$,
\begin{align*}
  &c_B^{1/p}\norm{\hat u_{\eps_1}-\hat u_{\eps_2}}_{L^p(0,\hat T; V)}\\
  &\leq
  \left(\norm{\hat v_{\eps_1}}_{L^1(0,\hat T; H)}^{1/p}
  +\norm{\hat v_{\eps_2}}_{L^1(0,\hat T; H)}^{1/p}\right)
  \norm{\hat u_{\eps_1}-\hat u_{\eps_2}}_{L^\infty(0,\hat T; H)}^{1/p}
  +(c_B')^{1/p}\norm{\hat u_{\eps_1}-\hat u_{\eps_2}}_{L^2(0,\hat T; H)}^{2/p}\,.
\end{align*}
Hence, recalling the estimate \eqref{est6_hat},
again by the H\"older inequality we have, for 
every $r\in[1,\frac{p^2\ell}{p+1})$, that
\begin{align*}
  \norm{\norm{\hat u_{\eps_1}-\hat u_{\eps_2}}_{L^p(0,\hat T; V)}}_{L^r(\Omega)}
  \leq  C  \left(\|\hat u_{\eps_1}-\hat u_{\eps_2}\|^{1/p}_{L^{r'}
  (\Omega;L^\infty(0,\hat T; H))}
  +\norm{\hat u_{\eps_1}-\hat u_{\eps_2}}_{L^{2r/p}(\Omega;L^2(0,\hat T; H))}^{2/p}\right)
\end{align*}
for some $r'\in[1,p\ell)$. Consequently, the 
convergence \eqref{conv8_hat} and the fact that $\frac{2r}p< p\ell$ ensure 
that \eqref{conv11_hat} holds.\\
{\it Proof of \eqref{conv15_hat}.}
Thanks to the strong convergences
\eqref{conv2_hat} and \eqref{conv11_hat} we have,
possibly along a subsequence, 
\[
  (\hat t_\eps,\hat u_\eps)
  \to (\hat t,\hat u) \qquad\text{in } [0,T]\times V
  \quad\text{a.e.~in } \Omega\times(0,\hat T)\,,
\]
from which it follows, thanks to the continuity assumption on $L$ in {\bf G}, that
\[
  \hat L_\eps(\cdot, \hat u_\eps)=
  L(\hat t_\eps, \hat u_\eps)\to 
  L(\hat t, \hat u)=\hat L(\cdot, \hat u)
  \quad\text{in } \cL^1(H,H)
  \quad\text{a.e.~in } \Omega\times(0,\hat T)\,.
\]
Furthermore, by the growth condition on $L$ in assumption
{\bf G},
 we get 
\[
  \norm{\hat L_\eps(\cdot, \hat u_\eps)}_{\cL^1(H,H)}
  \leq C_G(1+\norm{\hat u_\eps}_V^p)\,,
\]
which in turn implies, thanks to the estimate \eqref{est1_hat}, that 
\[
 \norm{\hat L_\eps(\cdot, \hat u_\eps)}_{
 L^\ell_{\hat 
 \cP}(\Omega\times(0,\hat T); \cL^1(H,H))}\leq C_\ell\,.
\]
Hence, the convergence \eqref{conv15_hat} follows from 
the Vitali Convergence Theorem.\\
{\it Proof of \eqref{conv12_hat}--\eqref{conv13_hat}--\eqref{conv14_hat}.}
By assumption {\bf B} we deduce that 
\[
  \norm{B(\hat u_\eps)-B(\hat u)}_{L^p(0,T; V^*)}
  \leq M\left(1+\norm{\hat u_\eps}_{L^\infty(0,T; V)}^{p-2}
  +\norm{\hat u}_{L^\infty(0,T; V)}^{p-2}
  \right)\norm{\hat u_\eps-\hat u}_{L^p(0,T; V)}\,,
\]
from which, taking \eqref{est1_hat} and \eqref{conv11_hat} into account, we deduce 
the convergence  \eqref{conv12_hat}. Clearly, 
thanks to the estimates \eqref{est5_hat} and \eqref{est7_hat},
 this also yields convergence  \eqref{conv13_hat}.
\end{proof}

Let us deal now with the measurability properties of the limiting processes.
Since $\hat t(\cdot,\tau)$ is the almost sure limit of $\hat t_\eps(\cdot,\tau)$
for every $\tau\geq0$, it follows from Lemma~\ref{lem:t} that 
$\hat t(\cdot, \tau)$ is a stopping time with respect to the original 
filtration $(\cF_t)_{t\in[0,T]}$ and that $\hat t$ is $\cP$-measurable.
Consequently,  the filtration
\[
  \cF_{\hat t(\cdot, \tau)}:=
  \left\{F\in\cF:\; F\cap\{\omega\in\Omega:\hat t(\omega,\tau)\leq t\}
  \in\cF_t\quad\forall\,t\in[0,T]\right\}\,
\]
is well defined. 
With this notation, we define the rescaled filtration as
\[
  \hat\cF_{\tau}:=  
  \sigma\{\hat u(\rho), \hat u^d(\rho), \hat M(\rho), \hat t(\rho) 
  :\rho\in[0,\tau]\}\,,\quad\tau\geq0\,,
  \qquad  \hat\cF:=\cF\,.
\]
The following result is the natural follow-up of Lemma~\ref{lem:hat} at the limit.
\begin{lem}
  \label{lem:hat_lim}
  The following  holds.
  \begin{enumerate}[(i)]
 \item\label{lem:hat1_lim} $(\hat\cF_{\tau})_{\tau\geq0}$
 is a filtration satisfying the usual conditions, with associated 
 progressive $\sigma$-algebra denoted by $\hat \cP$.
 \item\label{lem:hat2_lim} 
 $\hat t$, $\hat u$, and $\hat u^d$ are $\hat \cP$-progressively measurable
 on $(\Omega,\hat \cF,(\hat\cF_{\tau})_{\tau\geq0},\P)$, and satisfy 
 \begin{align*}
 \hat t'(\tau) + \norm{\partial_t\hat u^d(\tau)}_H\leq1
 \qquad&\text{for a.e.~}\tau>0\,,\quad\P\text{-a.s.}\,.
 \end{align*}
 \item\label{lem:hat3_lim} 
 $\hat M$ is a continuous square-integrable $U$-cylindrical martingale 
 on $(\Omega,\hat \cF,(\hat\cF_{\tau})_{\tau\geq0},\P)$, with 
 tensor quadratic variation given by 
 \[
 \qqv{\hat M}(\tau)=Q_1\hat t(\tau)\,, \qquad\tau\geq0\,.
 \]
 \item\label{lem:hat4_lim} $\hat G(\cdot, \hat u)$ is 
 stochastically integrable with respect to $\hat M$ and it holds for all $\hat T>0$ that 
 \beq
 \label{conv10_hat}
 \hat G_\eps(\cdot, \hat u_\eps)\cdot\hat M_\eps
 \to \hat G(\cdot, \hat u)\cdot\hat M
 \qquad\text{in } L^{r}_{\hat\cP}(\Omega; C^0([0,\hat T]; H))
 \quad\forall\,r\in[1,p\ell)\,.
 \eeq
 \end{enumerate} 
\end{lem}
\begin{proof}
 Ad~\ref{lem:hat1_lim}.  The fact that 
$(\hat\cF_{\tau})_{\tau}$ is a complete right-continuous 
filtration follows from the fact that $\hat t$ is 
a continuous increasing adapted process,
using the same arguments of the proof of Lemma~\ref{lem:hat}. \\
Ad~\ref{lem:hat2_lim}.  This follows directly 
from the definition of the rescaled filtration 
$(\hat\cF_\tau)_{\tau\geq0}$ itself and by weak 
lower semicontinuity in condition \eqref{est2_hat}. \\
Ad~\ref{lem:hat3_lim}. Clearly, $\hat M$ is adapted by 
definition of $(\hat\cF_\tau)_{\tau\geq0}$, and we have also already 
seen that it has continuous trajectories in $U_1$.
To show that it is a martingale with values in $U_1$, we note that 
by Lemma~\ref{lem:hat}, for every $\sigma,\tau\geq0$ with $\sigma\leq\tau$
it holds
\[
  \E[(\hat M_\eps(\tau)-\hat M_\eps(\sigma)) \psi((\hat u_\eps,
  \hat u^d_\eps,\hat M_\eps,\hat t_\eps)_{|[0,\sigma]}) ]
  \qquad
  \forall\,\psi\in C^0_b(C^0([0,\sigma];  H \times H \times
  U_1\times \Rz  ))\,.
\]
Letting $\eps\to0^+$ by \eqref{conv7_hat} and the Dominated Convergence Theorem, 
we deduce that 
\[
  \E[(\hat M(\tau)-\hat M(\sigma))  \psi((\hat u_\eps,
  \hat u^d_\eps,\hat M_\eps,\hat t_\eps)_{|[0,\sigma]})] =0 \qquad
  \forall\,\psi\in C^0_b(C^0([0,\sigma];  H \times H \times
  U_1\times \Rz  ))\,,
\]
from which it follows that $\hat M$ is a martingale
with respect to its natural filtration, hence also 
with respect to $(\hat\cF_\tau)_{\tau\geq0}$.
Moreover, by a similar argument, from Lemma~\ref{lem:hat}
we know that for any $\sigma,\tau\geq0$ with $\sigma\leq\tau$,
for any $\varphi_1,\varphi_2\in U_1$
it holds that
\begin{align*}
  &\E\Big[\Big((\hat M_\eps(\tau), \varphi_1)_{U_1}(\hat M_\eps(\tau), \varphi_2)_{U_1}
  -(\hat M_\eps(\sigma), \varphi_1)_{U_1}(\hat M_\eps(\sigma), \varphi_2)_{U_1}\\
  &\qquad
  -(\hat t_\eps(\tau)-\hat t_\eps(\sigma))(Q_1\varphi_1, \varphi_2)_{U_1}
  \Big)  \psi((\hat u_\eps,
    \hat u^d_\eps,\hat M_\eps,\hat t_\eps)_{|[0,\sigma]}) \Big]=0 \\
  &
  \qquad \qquad\forall\,\psi\in C^0_b(C^0([0,\sigma];  H \times H \times
  U_1\times \Rz  ))\,.
\end{align*}
Again by \eqref{conv2_hat}, \eqref{conv7_hat}
and the Dominated Convergence Theorem this yields
that
\begin{align*}
  &\E\Big[\Big((\hat M(\tau), \varphi_1)_{U_1}(\hat M(\tau), \varphi_2)_{U_1}
  -(\hat M(\sigma), \varphi_1)_{U_1}(\hat M(\sigma), \varphi_2)_{U_1}\\
  &\qquad
  -(\hat t(\tau)-\hat t(\sigma))(Q_1\varphi_1, \varphi_2)_{U_1}
  \Big)  \psi((\hat u_\eps,
  \hat u^d_\eps,\hat M_\eps,\hat t_\eps)_{|[0,\sigma]}) \Big]=0 \\
  &
  \qquad \qquad\forall\,\psi\in C^0_b(C^0([0,\sigma];  H \times H \times
  U_1\times \Rz  ))\,,
\end{align*}
from which it follows that the tensor quadratic variation of $\hat M$ is 
exactly $Q_1\hat t$.\\
Ad~\ref{lem:hat4_lim}. Clearly, $\hat G(\cdot, \hat u)$ is $\hat\cP$-measurable 
since so are $\hat u$  and  $\hat t$. 
Moreover, one has by the regularity of $\hat G(\cdot, \hat u)$ that
\begin{align*}
  \E\int_0^\tau\|\hat G(\sigma, \hat u(\sigma))\|_{\cL^2(U,H)}^2\,\d\hat t(\sigma)&=
  \E\int_0^\tau\|\hat G(\sigma, \hat u(\sigma))\|_{\cL^2(U,H)}^2
  \hat t'(\sigma)\,\d\sigma\\
  &\leq
  \E\int_0^{\tau}\|\hat G(\sigma, \hat u(\sigma))\|_{\cL^2(U,H)}^2\,\d\sigma
  <+\infty\,,
\end{align*}
so that $\hat G(\cdot, \hat u)$ is stochastically integrable with respect to $\hat M$. 
Eventually, following exactly the same computations 
based on the Burkholder--Davis--Gundy inequality 
and assumption {\bf G}
used after \eqref{conv8_aux} in order to prove \eqref{conv8_hat},
it is not difficult to see that, for all $r\in[1,p\ell)$,
\begin{align*}
  &\norm{\hat G_{\eps}(\cdot, \hat u_{\eps})\cdot\hat M_{\eps}
  -\hat G(\cdot, \hat u)\cdot\hat M
  }_{L^r_\cP(\Omega; C^0([0,\tau]; H))}^r\\
  &\leq
   C  \E\left(\int_0^\tau\norm{\hat G_{\eps}(\sigma, \hat u_{\eps}(\sigma))
  -\hat G(\sigma, \hat u(\sigma))}_{\cL^2(U,H)}^2
  \hat t'_{\eps}(\sigma)\,\d\sigma\right)^{r/2}\\
  &\qquad+ C \E\left(\int_0^\tau\norm{
  \hat G(\sigma, \hat u(\sigma))}_{\cL^2(U,H)}^2
  \,\d\qv{\hat M_{\eps}-\hat M}(\sigma)\right)^{r/2}\\
  &\leq  C \norm{\hat G_\eps(\cdot, \hat u_\eps)-
  \hat G(\cdot, \hat u)}_{L^r_\cP(\Omega; L^2(0,\hat T; \cL^2(U,H)))}\\
  &\qquad + C_{\hat T}  \E\bigg[
  \norm{\hat G(\cdot, \hat u)}_{C^0([0,\hat T]; \cL^2(U,H))}^r
  \qv{\hat M_\eps-\hat M}(\tau)^{r/2}\bigg]\,,
  \end{align*}
  where, by the H\"older and Burkholder--Davis--Gundy inequalities,
  \begin{align*}
  &\E\bigg[
  \norm{\hat G(\cdot, \hat u)}_{C^0([0,\hat T]; \cL^2(U,H))}^r
  \qv{\hat M_\eps-\hat M}(\tau)^{r/2}\bigg]\\
  &\qquad\leq 
  \norm{\hat G(\cdot, \hat u)}^r_{L^{p\ell}_\cP(\Omega; C^0([0,\hat T]; \cL^2(U,H)))}
  \norm{\hat M_\eps-\hat M}^r_{L^{\frac{rp\ell}{p\ell-r}}_{\cP}(\Omega; 
  C^0([0,\hat T]; U_1))}\,.
  \end{align*}
  From the regularity of $\hat G(\cdot, \hat u)$ we infer that
  \begin{align*}
    &\norm{\hat G_{\eps}(\cdot, \hat u_{\eps})\cdot\hat M_{\eps}
  -\hat G(\cdot, \hat u)\cdot\hat M
  }_{L^r_\cP(\Omega; C^0([0,\tau]; H))}\\
  &\leq  C_{\hat T} \left[
  \norm{\hat G_\eps(\cdot, \hat u_\eps)-
  \hat G(\cdot, \hat u)}_{L^r_\cP(\Omega; L^2(0,\hat T; \cL^2(U,H)))}+
  \norm{\hat M_\eps-\hat M}_{L^{\frac{rp\ell}{p\ell-r}}_{\cP}(\Omega; 
  C^0([0,\hat T]; U_1))}
  \right]\,,
  \end{align*}
  and thanks to the convergences \eqref{conv9_hat} and \eqref{conv7_hat} we conclude.
\end{proof}

We conclude this subsection on the convergence of  the  rescaled problems 
by highlighting the asymptotic behaviour of the functionals $\Psi_\eps$
in the sense of Mosco convergence.  We record here the following
lemma, which we present by omitting the elementary proof  \cite{Attouch}.
\begin{lem}
  \label{lem:mosco}
  It holds that $\widehat\Psi_\eps\to\Psi_{\norm{\cdot}}$ in the sense of Mosco, i.e.
  \begin{align*}
    &\forall\,x\in H\,,\quad\Psi_{\norm{\cdot}}(x)\leq\liminf_{\eps\to0^+}
    \widehat\Psi_\eps(x_\eps)
    \quad\forall\,(x_\eps)_\eps\subset H \text{ with } x_\eps\wto x \text{ in } H\,,\\
    &\forall\,x\in H\,, \quad\exists\,(x_\eps)_\eps\subset H:\quad
    x_\eps\to x\;\text{in } H \quad\text{ and }\quad \lim_{\eps\to0^+}
    \widehat\Psi_\eps(x_\eps)=\Psi_{\norm{\cdot}}(x)\,.
  \end{align*}
  This implies that 
  $\partial\widehat\Psi_\eps \to \partial\Psi_{\norm{\cdot}}$
  in the graph sense, i.e.
  \[
  \forall\,x,y\in H\,, \quad\exists\,(x_\eps)_\eps,(y_\eps)_\eps\subset H:
  \quad x_\eps\to x \text{ in } H\,,
  \quad y_\eps\to y \text{ in } H\,,
  \quad y_\eps\in\partial\widehat\Psi_\eps(x_\eps)\quad\forall\,\eps\,.
  \]
\end{lem}
% \begin{proof}
% This follows immediately from the definition of Mosco convergence
% and $\mathcal F$.
% \end{proof}

\subsection{Passage to the limit in the rescaled problem}\label{sec:pa}
The convergences  \eqref{conv2_hat}--\eqref{conv7_hat}
and the Lemmas~\ref{lem:strong_conv}--\ref{lem:hat_lim} ensure, \UUU by \EEE
letting $\eps\searrow0$ in \eqref{eq1_hat_eps}--\eqref{eq2_hat_eps}, that
\beq
  \label{eq1_hat}
  \hat u(\tau) = \hat u^d(\tau)
  +\int_0^{\tau}\hat G(\sigma,\hat u(\sigma))\,\d\hat M(\sigma)
  \qquad\forall\,\tau\geq0\,, \quad \P\text{-a.s.}
\eeq
and
\beq
  \label{eq2_hat}
  \hat v(\tau) + B(\hat u(\tau)) = 0 \qquad\text{for a.e.~}\tau>0\,, \quad \P\text{-a.s.}
\eeq
In order to conclude, we need to identify the limit $\hat v$. To this end, we use a
lower semicontinuity argument in the energy identity and exploit the Mosco convergence
 of  Lemma~\ref{lem:mosco}.
 Note first that 
from the energy identity \eqref{en_eps},
the rescaling process $\hat t_\eps$
and the change-of-variable formula in Lemma~\ref{lem:hat}
yield the rescaled energy identity, for every $\tau>0$,
\begin{align}
\label{en_id_eps}  
  \E\Phi( \hat u_\eps(\tau)) 
  +\E\int_0^\tau\left( \hat v_\eps(\sigma), 
  \partial_t \hat u^d_\eps(\sigma)\right)_H\,\d \sigma 
  &=\Phi(u_0) 
  +\frac12\E\int_0^\tau
  \operatorname{Tr}\left[\hat L_\eps(\sigma, \hat u_\eps(\sigma))\right]
  \hat t_\eps'(\sigma)\,\d \sigma\,.
\end{align}
%\RRR
%Since the last contribution is a $(\hat\cF_{\eps,\tau})_\tau$-martingale,
%this implies in particular that 
%for every $\tau_1,\tau_2>0$ with $\tau_1<\tau_2$
%and for every nonnegative $Y_\eps\in L^\infty(\Omega,\hat\cF_{\eps, \tau_1})$
%it holds that 
%\begin{align}
%\nonumber
%&\E Y_\eps\Phi(\hat u_\eps(\tau_2)) + 
%\E Y_\eps\int_{\tau_1}^{\tau_2}(\hat v_\eps(\sigma), 
%\partial_t\hat u^d_\eps(\sigma))_H\,\d\sigma\\
%&\qquad\label{en_id_eps2}
%=\E Y_\eps\Phi(\hat u_\eps(\tau_1)) 
%+ \frac12\E Y_\eps\int_{\tau_1}^{\tau_2}
%\operatorname{Tr}[\hat L(\sigma, \hat u_\eps(\sigma))]
%\hat t_\eps'(\sigma)\,\d\sigma\,.
%\end{align}
%\EEE

The identification of the nonlinearity $\hat v$
calls for the corresponding exact energy identity for the limiting 
rescaled processes, which is not trivial. Indeed, the It\^o formula 
for $\Phi$ at the limit rescaled problem does not fall neither in the classical setting,
due to the fact that $\Phi$ is only defined on $V$, nor in the extended 
version \cite[Prop.~3.3]{ScarStef-SDNL2}, due to the 
poor time-regularity of $B(\hat u)$ and the general Hilbert-space-valued
martingale noise. We show here that nonetheless
it is possible to write an exact energy equality at the rescaled limit.

\begin{prop}[It\^o formula in the limit]
\label{prop:ito}
In the current setting, it holds for every $\tau>0$ that 
\begin{align}
\label{en_id_lim}
\E \Phi(\hat u(\tau)) + 
\E\int_0^{\tau}(\hat v(\sigma), 
\partial_t\hat u^d(\sigma))_H\,\d\sigma
=\Phi(u_0) 
+ \frac12\E \int_0^\tau
\operatorname{Tr}[\hat L(\sigma, \hat u(\sigma))]
\hat t'(\sigma)\,\d\sigma\,.
\end{align}
\end{prop}
\begin{proof}[Proof of Proposition~\ref{prop:ito}]
\UUU
  Let $\hat T>0$ and let $(P_n)_n$ be as in assumption {\bf B}.
  Using the subscript $n$ to denote the action of $P_n$, 
  one has from \eqref{eq1_hat} that
  \[
  \hat u_n(\tau) = \hat u^d_n(\tau)
  +\int_0^{\tau}\hat G_n(\sigma,\hat u(\sigma))\,\d\hat M(\sigma)
  \qquad\forall\,\tau\geq0\,, \quad \P\text{-a.s.}
  \]
  where $\hat u_n\in L^{p\ell}_\cP(\Omega; L^\infty(0,\hat T; Z))$
  by definition of $P_n$ and the regularity of $\hat u$. This implies that 
  $B(\hat u_n)\in L^{q\ell}_\cP(\Omega; L^\infty(0,\hat T; H))$
  again by assumption {\bf B}. Hence, 
  an elementary adaptation of the It\^o formula \cite[Lem.~3.2]{ScarStef-SDNL2} 
to the case of continuous square-integrable martingale noise yields
\begin{align*}
&\E\Phi(\hat u_n(\tau)) 
-\E\int_0^\tau(B(\hat u_n(\sigma)), \partial_t\hat u_n^d(\sigma))_H\,\d\sigma\\
&=\Phi(u_{0,n}) 
+ \frac12\E\int_0^\tau
\operatorname{Tr}[\hat G_n(\sigma, \hat u(\sigma))^*
D_{\mathcal G}B(\hat u_n(\sigma))
\hat G_n(\sigma, \hat u(\sigma))]
\hat t'(\sigma)\,\d\sigma\,.
\end{align*}
Since $B(\hat u)\in L^\ell_\cP(\Omega; L^1(0,\hat T; H))$,
one has from assumptions {\bf B--G}
and the dominated convergence theorem 
that
\begin{align*}
  \hat u_n \to \hat u \quad&\text{in } 
  L^{\ell p}_\cP(\Omega; L^{r}(0,\hat T; V))\quad\forall\,r\geq1\,,\\
  B(\hat u_n) \to B(\hat u) \quad&\text{in } L^{\ell}_\cP(\Omega; L^1(0,\hat T; H))\,,\\
  \hat u_n(\sigma)\to \hat u(\sigma) \quad&\text{in } L^{\ell p}_\cP(\Omega; V)
  \quad\forall\,\sigma\geq0\,,\\
  \hat G_n(\cdot, \hat u) \to \hat G(\cdot, \hat u)
  \quad&\text{in } L^{\ell}_\cP(\Omega; L^2(0,\hat T; \cL^2(U,V)))\,.
\end{align*}
This implies by the weak continuity of $D_{\mathcal G}B$ in assumption {\bf B} that 
\[
  D_{\mathcal G}B(\hat u_n)\to D_{\mathcal G}B(\hat u) \quad\text{in } \cL_w(V,V^*)\,,
  \quad\text{a.e.~in } \Omega\times(0,\hat T)\,,
\]
hence also from assumption {\bf G} and the dominated convergence theorem that 
\[
  \operatorname{Tr}[\hat G_n(\cdot, \hat u)^*
D_{\mathcal G}B(\hat u_n)
\hat G_n(\cdot, \hat u)]
\hat t'(\cdot) \to 
\operatorname{Tr}[\hat L(\cdot, \hat u)]
\hat t'(\cdot) \quad\text{in } L^1_\cP(\Omega; L^1(0,\hat T))\,.
\]
Moreover, recalling that $\partial_t\hat u^d\in 
L^\infty_\cP(\Omega\times(0,+\infty); H)$ we also have that 
\[
  \partial_t\hat u^d_n \wstarto 
  \partial_t\hat u^d \quad\text{in } L^\infty_\cP(\Omega\times(0,\hat T); H)\,.
\] 
Hence, we can pass to the limit in all terms as $n\to\infty$ and conclude
from equation \eqref{eq2_hat}.\EEE
\end{proof}

We are now in the position of arguing on
the energy identity \eqref{en_id_eps} by lower semicontinuity.
%\RRR To this end, let first $\tau_1,\tau_2>0$ with $\tau_1<\tau_2$
%be taken arbitrary in a full measure set 
%such that the convergence \eqref{conv11_hat} holds pointwise 
%in $\tau_1$ and $\tau_2$ (possibly along subsequences), namely 
%\[
%  u_\eps(\tau_i)\to u(\tau_i) \qquad\text{in } L^p(\Omega; V)\,,\qquad
%  i=1,2\,.
%\]
%Moreover, let 
%$Y\in L^\infty(\Omega,\hat\cF_{\tau_1})$ nonnegative be arbitrary,
%and define for every $\eps>0$
%\[
%  \hat\cG_{\eps,\tau_1}:=\hat\cF_{\tau_1}\cap\bigvee_{\tilde\eps\in[\eps,1]}
%  \hat\cF_{\tilde\eps,\tau_1}\,, \qquad
%  Y_\eps:=\E\left[Y\Big|\hat\cG_{\eps,\tau_1}\right]\in 
%  L^\infty(\Omega,\hat\cF_{\eps,\tau_1})\,.
%\]
%Since by definition of $\hat\cF_{\tau_1}$ it holds
%\[
%  \hat\cF_{\tau_1}\subseteq\bigvee_{\eps>0}\hat\cF_{\eps,\tau_1}\,,
%\]
%we have in particular that $\hat\cG_{\eps,\tau_1}\nearrow\hat\cF_{\tau_1}$ as $\eps\to0^+$
%\[
%  Y_\eps\to Y \qquad\text{a.e.~in } \Omega\,.
%\]
%\EEE
To this end,
thanks to the convergences \eqref{conv3_hat}, 
\eqref{conv8_hat}, \eqref{conv15_hat}, and
the lower semicontinuity of $\Phi$,
 for every  $\tau >0$  
we infer that
\begin{align*}
  &\limsup_{\eps\to0^+}
  \E\int_{0}^{\tau}\left( \hat v_\eps(\sigma), 
  \partial_t \hat u^d_\eps(\sigma)\right)_H\,\d \sigma \\
  &\leq \limsup_{\eps\to0^+}
  \E\Phi(  u_0 ) 
  +\frac12\limsup_{\eps\to0^+}\E\int_{0}^{\tau}
  \operatorname{Tr}
  \left[\hat L_\eps(\sigma, \hat u_\eps(\sigma))\right]\hat t_\eps'(\sigma)
  \,\d \sigma
  -\liminf_{\eps\to0^+}\E\Phi( \hat u_\eps( \tau)))\\
  &\leq\E\Phi(  u_0 )
  +\frac12\E\int_{0}^{\tau}
  \operatorname{Tr}\left[\hat L(\sigma, \hat u(\sigma))\right]\hat t'(\sigma)\,\d \sigma
  -\E\Phi( \hat u( \tau))\,.
\end{align*}
 where we have used the pointwise weak convergence 
\eqref{conv13_hat_due} of $\hat u_\eps$ as well as the weak lower semicontinuity  of
$\Phi$ in $V$. 
The above right-hand side can in turn be rewritten thanks to the limit energy identity \eqref{en_id_lim} as 
\begin{equation}
  \limsup_{\eps\to0^+}
  \E \int_0^\tau  \left( \hat v_\eps(\sigma), 
  \partial_t \hat u^d_\eps(\sigma)\right)_H\,\d \sigma
  \leq 
  \E \int_0^\tau  \left( \hat v(\sigma), 
  \partial_t \hat u^d(\sigma)\right)_H\,\d \sigma
  \qquad\forall\,0\leq \tau \, . 
\label{eq:limsup}
\end{equation}

In order to conclude for $\hat v\in
\partial\Psi_{\norm{\cdot}}(\partial_t\hat u)$ a.e. in
$\Omega\times\erre_+$, we recall that
\begin{equation}
  \E\int_0^\tau\left( \hat v_\eps(\sigma) , 
 \tilde w (\sigma)- \partial_t \hat u^d_\eps(\sigma) \right)_H \,\d \sigma \leq \E\int_0^\tau\haz
\Psi_\eps (\tilde w (\sigma)) \,\d \sigma- \E\int_0^\tau\haz \Psi_\eps (\partial_t \hat
u^d_\eps(\sigma))\,\d \sigma\label{eq:sotto}
\end{equation}
for $\tilde w \in L^2_{\hat \cP} (\Omega; L^2(0,T;H))$. Fix now $w \in L^2_{\hat \cP}
(\Omega; L^2(0,T;H))$ such that $\Phi_{\| \cdot \|}(w) <+\infty$  a.e. in
  $\Omega\times\erre_+$.  For all 
$\mu>0$ let
$w^\mu$ be the unique solution of the singular perturbation problem
$$w^\mu + \mu^2 F(w^\mu) \ni w \ \ \quad
\text{a.e. in}\ \ \Omega\times\erre_+\, $$
where $F: V \to 2^{V^*}$ stands for the duality map. 
Note that $w^\mu \in   L^2_{\hat \cP} (\Omega; L^2(0,T;V))$ and $\| w^\mu
\|_H \leq \| w\|_H$ a.e. in
  $\Omega\times\erre_+$. Choose now $\tilde w=w^\mu$ in inequality
  \eqref{eq:sotto} and pass to the limsup for $\eps\to 0$. Also using
  convergences \eqref{conv6_hat} and \eqref{conv14_hat},
  the  Mosco-convergence Lemma~\ref{lem:mosco} and the limsup
  inequality \eqref{eq:sotto2} we get 
\begin{equation}
  \E\int_0^\tau\left( \hat v (\sigma) , 
 w^\mu(\sigma) - \partial_t \hat u^d(\sigma) \right)_H \,\d \sigma \leq \E\int_0^\tau
\Psi_{\|\cdot \|} (w^\mu(\sigma)) \,\d \sigma- \E\int_0^\tau \Psi_{\|\cdot \|} (\partial_t \hat
u^d(\sigma))\,\d \sigma\,.\label{eq:sotto2}
\end{equation}
By passing to the limit as $\mu \to 0$ in the latter, recalling
that $w^\mu \to w $ strongly in $L^2_{\hat \cP} (\Omega; L^2(0,T;H))$, we
prove that indeed \EEE
\[
\hat v\in \partial\Psi_{\norm{\cdot}}(\partial_t\hat u) \quad\text{a.e.~in } 
\Omega\times\erre_+\,.
\]

In  order to conclude we are left to show that 
the relation 
\eqref{eq_hat} holds with the equality, and  not only 
with the inequality as  in Lemma~\ref{lem:hat_lim}.
To this end, given
any $\tilde w \in  L^2_{\hat \cP} (\Omega; L^2(0,T;V))$ and any 
$0< \tau $,
by \eqref{eq:limsup} we have that
\begin{align}
  &\limsup_{\eps\to0^+} \E \int_0^\tau  
  \Psi_{\|\cdot \|} (\partial_t \hat
    u^d_\eps(\sigma)) \,\d \sigma \nonumber\\
  &\quad \leq 
  \limsup_{\eps\to0^+}\Bigg( \E \int_0^\tau  
  \Psi_{\|\cdot
    \|} (\tilde w(\sigma)) \,\d \sigma
     -  \E \int_0^\tau  
  (\hat v_\eps(\sigma), \tilde w(\sigma) - \partial_t
  \hat u^d_\eps(\sigma))_H\,\d \sigma\Bigg) \nonumber\\
  & \quad  \leq  \E \int_0^\tau 
  \Psi_{\|\cdot
    \|} (\tilde w(\sigma)) \,\d \sigma
  - \E \int_0^\tau 
     (\hat v(\sigma) , \tilde w(\sigma) - \partial_t
  \hat u^d (\sigma))_H\,\d \sigma\,. \label{eq:sotto3}
\end{align}
By density, the same inequality holds for all $\tilde w \in  L^2_{\hat \cP}
(\Omega; L^2(0,T;H))$. Hence, one can choose $\tilde w = \partial_t \hat
u^d$ in \eqref{eq:sotto3} in order to get that 
\beq\label{eq:sotto4}
\limsup_{\eps\to0^+} 
\E \int_0^\tau  
\| \partial_t \hat u^d_\eps(\sigma) \|_H\,\d\sigma
\leq 
\E \int_0^\tau  
\| \partial_t \hat u^d(\sigma) \|_H \, \d \sigma
\qquad\forall\, 0\leq\tau\,.  
\eeq
At this point, 
integrating the relation \eqref{est2_hat}
$\hat t'_\eps + \|\partial_t \hat u_\eps\|_H=1$ we obtain that 
\[
  \E\left[\frac{\hat t_\eps( \tau ) }{ \tau } \right]
  + \frac1{ \tau }\E\int_0^\tau 
  \norm{\partial_t\hat u_\eps(\sigma)}_H\,\d \sigma=1
  \qquad\forall\,0\leq \tau \,,
\]
so that \eqref{eq:sotto4} yields, letting $\eps\to0^+$,
\[
  \E\left[\frac{\hat t( \tau ) }{ \tau } \right]
  + \frac1{ \tau }\E\int_0^\tau 
  \norm{\partial_t\hat u(\sigma)}_H\,\d \sigma\geq 1
  \qquad\forall\,0\leq \tau \,.
\]
 In particular, we have proved that
\[ \E\int_0^\tau  (\hat t '(\sigma) + 
  \norm{\partial_t\hat u(\sigma)}_H-1)\,\d \sigma\geq 0
  \qquad\forall\,0\leq\UUU \tau \,.
\]
On the other hand,  by Lemma~\ref{lem:hat_lim} we already know that 
$ \hat t' +\norm{\partial_t \hat u}_H-1\leq0$ almost everywhere,
this necessarily implies that 
\[
  \hat t'(\tau) + \norm{\partial_t \hat u(\tau)}_H =1 
  \qquad\text{for a.e.~}\tau>0\,, \quad\P\text{-a.s.}
\]
We have hence checked that 
  \eqref{eq_hat} is actually an
  equality. This concludes the proof of Theorem~\ref{thm1}.

%%%%%%%%%%%%%%%%%%%%%%%%%%%%%%%%%%%%%%

\section*{Acknowledgement}
LS was funded by the Austrian Science Fund (FWF) through the
Lise-Meitner project M\,2876. US is partially supported by 
the Austrian Science Fund (FWF) through projects  F\,65, W\,1245,  I\,4354, I\,5149, and P\,32788, and by the OeAD-WTZ project CZ
01/2021. 
\UUU
LS is member of Gruppo Nazionale 
per l'Analisi Matematica, la Probabilit\`a 
e le loro Applicazioni (GNAMPA), 
Istituto Nazionale di Alta Matematica (INdAM), 
and gratefully acknowledges financial support 
through the project CUP\_E55F22000270001.
The present research has also been supported by 
the MUR grant  ``Dipartimento di
eccellenza 2023--2027'' for Politecnico di Milano. The careful
reading and the interesting remarks by an anonymous referee on a previous
version of the paper are gratefully acknowledged. 
\EEE

%%%%%%%%%%%%%%%%%%%%%%%%%%%%%%%%%%%%%%

%\bibliography{ref}{}
\bibliographystyle{abbrv}

%\end{document}

\end{document}